\documentclass[a4paper,11pt]{amsart}
\allowdisplaybreaks 
\usepackage[english]{babel}
\usepackage[utf8]{inputenc}
\usepackage[T1]{fontenc}
\usepackage[%
	useprefix,%
	hyperref,%
	backend=bibtex,%
	maxbibnames=99,%
	url=false,%
	doi=false,%
	isbn=false,%
	firstinits=true
	]{biblatex}
\bibliography{./Some-remarks-on-contact-variations-BIB}

\usepackage{amssymb}
\usepackage{mathrsfs}

\usepackage{hyperref}
\usepackage{csquotes}
\usepackage{enumitem}	

\newcommand{\scr}[1]{\mathscr{#1}}
\newcommand{\frk}[1]{\mathfrak{#1}}
\newcommand{\bb}[1]{\mathbb{#1}}

\newcommand{\N}{\mathbb{N}}	
\newcommand{\R}{\mathbb{R}}	
\newcommand{\Co}{\mathscr{C}}	
\renewcommand{\Vec}{\mathrm{Vec}}	
\newcommand{\Id}{\mathrm{Id}}	
\newcommand{\Span}{\mathrm{span}}	
\newcommand{\ssubset}{\subset\subset}	
\newcommand{\dd}{\,\mathrm{d}}	
\newcommand{\de}{\partial}		
\newcommand{\LL}{{\scr L}}	
\renewcommand{\div}{{\rm div}}	

\newcommand{\HH}{\mathbb{H}}	
\newcommand{\hh}{\frk h} 
\newcommand{\grad}{\nabla}
\newcommand{\spt}{{\rm spt}}
\newcommand{\THEN}{\Rightarrow}	

\newcommand{\Cow}{\Co^1_{\bb W}}
\newcommand{\II}{{I\!I}}

\renewcommand{\Psi}{{G}}

\theoremstyle{plain}
\newtheorem{proposition}{Proposition}[section]
\newtheorem{theorem}[proposition]{Theorem}
\newtheorem{lemma}[proposition]{Lemma}

\newtheorem{Coj}[proposition]{Conjecture}

\theoremstyle{definition}
\newtheorem{definition}[proposition]{Definition}
\newtheorem{remark}[proposition]{Remark}

\title[Contact variations in $\HH$]{%
	Some remarks on contact variations\\ in the first Heisenberg group}
\author[Golo]{Sebastiano Golo}
\thanks{The author has been supported by the People Programme (Marie Curie Actions) of the European Union's Seventh Framework Programme FP7/2007-2013/ under REA grant agreement n. 607643.}
\address[Golo]{University of Jyvaskyla, Finland; University of Trento, Italy}
\date{\today}

\subjclass[2010]{%
53C17,
49Q20
}
\keywords{%
Sub-Riemannian Geometry, %
Sub-Riemannian perimeter, %
Heisenberg Group, %
Bernstein's Problem, %
Contact diffeomorphisms%
}

\begin{document}
\maketitle
\begin{abstract}
 	We show that
	in the first sub-Riemannian Heisenberg group there are intrinsic graphs of smooth functions 
	that are both critical and stable points of the sub-Riemannian perimeter under compactly supported variations of contact diffeomorphisms,
	despite the fact that they are not area-minimizing surfaces.
	In particular, we show that if $f:\R^2\to\R$ is a $\Co^1$-intrinsic function, and $\grad^f\grad^ff=0$, then 
	the first contact variation of the sub-Riemannian area of its intrinsic graph is zero and 
	the second contact variation is positive.
\end{abstract}
\setcounter{tocdepth}{1}
\tableofcontents

\section{Introduction}

We want to address some new features of the sub-Riemannian perimeter in the Heisenberg group.
The notion of sub-Riemannian perimeter in the Heisenberg group, the so-called \emph{intrinsic perimeter}, has been enstablished as a direct and natural extension from the Euclidean perimeter in $\R^n$.
However, in many aspects, there are fundamental differences that lead to new open questions \cite{MR1871966,MR2313532,MR3587666,MR1404326,MR3363670}.

Before a detailed explanation, let us introduce some basic notions and notations we need in this introduction.
The (first) Heisenberg group $\HH$ is a three dimensional Lie group diffeomorphic to $\R^3$.
However, when endowed with a left-invariant sub-Riemannian distance, it becomes a metric space with Hausdorff dimension equal to four; see \cite{MR2312336}.

By standard methods of Geometric Measure Theory, one defines the \emph{intrinsic perimeter} $P(E;\Omega)$ of a measurable set $E\subset\HH$ in an open set $\Omega\subset\HH$.
We will denote it also by $\scr A(\de E\cap \Omega)$.

\emph{Regular surfaces} are topological surfaces in $\HH$ that admit a continuously varying tangent plane 
and they play an important role in the theory of sets with finite intrinsic perimeter.
They are the sub-Riemannian counterpart of smooth hypersurfaces in $\R^n$.
Regular surfaces are locally graphs of so-called \emph{$\Co^1$-intrinsic functions} $\R^2\to\R$.
We will focus on these functions and their graphs.

The space of all $\Co^1$-intrinsic functions will be denoted by $\Cow$ and the graph of $f:\R^2\to\R$ by $\Gamma_f\subset\HH$.
It is well known that $f\in\Cow$ if and only if $f\in\Co^0(\R^2)$ and the distributional derivative 
\[
\grad^ff = \de_\eta f+\frac12\de_\tau (f^2)
\]
is continuous, where we denote by $(\eta,\tau)$ the coordinates on $\R^2$; see \cite{MR2223801,MR3587666}.
If $\omega\subset\R^2$, the intrinsic area of $\Gamma_f$ above $\omega$ is
\[
\scr A(\Gamma_f\cap \Omega_\omega ) = \int_\omega \sqrt{1+(\grad^ff)^2} \dd \eta\dd \tau ,
\]
where $\Omega_\omega=\{(0,\eta,\tau)*(\xi,0,0):(\eta,\tau)\in\omega,\ \xi\in\R\}$, with $*$ denoting the group operation in $\HH$.

An important open problem concerning $\Cow$ is \emph{Bernstein's problem}: If the graph $\Gamma_f$ of $f\in\Cow$ is a locally minimizer of the intrinsic area, is $\Gamma_f$ a plane?
See Section~\ref{subs11241814} for a precise statement and \cite{MR2472175,MR2333095,MR2455341,MR3406514} for further reading.

In the study of perimeter minimizers in $\HH$, we identify three main issues that mark the gap from the Euclidean theory.
First, the map $f\mapsto\grad^ff$ is a nonlinear operator. 
Such non-linearity reflects on the fact that basic function spaces like $\Cow$ itself, or the space of functions with bounded intrinsic variation, are not vector spaces.
See Remark~\ref{rem10261128} for details.

Second, the area functional is not convex (say on $\Co^1(\R^2)$).
In particular, there are critical points that are not extremals, see \cite{MR2472175}.
In other words, a first variation condition
\begin{equation}\label{eq10261829}
\left.\frac{\dd}{\dd \epsilon}\right|_{\epsilon=0} \scr A(\Gamma_{f+\epsilon\phi}\cap (\Omega_\omega)) = 0
\qquad\forall\phi\in\Co^\infty_c(\omega)
\end{equation}
does not characterize minimizers.
However, if $f\in\Co^1(\R^2)$, a second variation condition $\left.\frac{\dd^2}{\dd \epsilon^2}\right|_{\epsilon=0} \scr A(\Gamma_{f+\epsilon\phi}\cap (\Omega_\omega)) \ge 0$ does, see \cite{MR3406514}.

Third, there are objects among sets of finite intrinsic perimeter with very low regularity, see Remark~\ref{rem10261128}.
The standard variational approach as in~\eqref{eq10261829} fails when applied to these objects.
More precisely, if $f\in\Cow$, 
then $\scr A(\Gamma_{f+\epsilon\phi}\cap (\Omega_\omega))$ may be $+\infty$ for all $\epsilon\neq0$, all $\omega\subset\R^2$ open and all $\phi\in\Co^\infty_c(\omega)\setminus\{0\}$.
In another approach, one can consider smooth one-parameter families of diffeomorphisms $\Phi_\epsilon:\HH\to\HH$ with $\Phi_0=\Id$ and $\{\Phi_\epsilon\neq\Id\}\ssubset\Omega$, and take variations of $\scr A(\Phi_\epsilon(\Gamma_f)\cap\Omega)$.
However, it may happen again that $\scr A(\Phi_\epsilon(\Gamma_f)\cap\Omega)=+\infty$ for all $\epsilon\neq0$.

After further considerations, one understands that we need to restrict the choice of $\Phi_\epsilon$ to \emph{contact diffeomorphisms}, see Proposition~\ref{prop11202024}.
See also~\cite{MR2312336} and~\cite{MR1317384} for references on contact diffeomorphisms.
In this setting, we address the question whether, despite this restriction, conditions on the first and second variations with contact diffeomorphisms can single out minimal graphs.
Our answer is no:
\begin{theorem}\label{teo10261843}
	There is $f\in\Cow$ such that, for all $\Omega\subset\HH$ open and all smooth one-parameter families of contact diffeomorphisms $\Phi_\epsilon:\HH\to\HH$ with $\Phi_0=\Id$ and $\{\Phi_\epsilon=\Id\}\ssubset\Omega$,
	it holds
	\[
	\left.\frac{\dd}{\dd \epsilon}\right|_{\epsilon=0} \scr A(\Phi_\epsilon(\Gamma_f)\cap \Omega) = 0
	\quad\text{ and }\quad
	\left.\frac{\dd^2}{\dd \epsilon^2}\right|_{\epsilon=0} \scr A(\Phi_\epsilon(\Gamma_f)\cap \Omega) \ge 0,
	\]
	but $\Gamma_f$ is not an area-minimizing surface.
\end{theorem}

In fact, all smooth solutions to the equation $\grad^f(\grad^ff)=0$ are examples of the functions appearing in the theorem.

The proof of Theorem~\ref{teo10261843} is based on a ``Lagrangian'' approach to $\Cow$.
Indeed, a function $f\in\Cow$ is uniquely characterized by the integral curves of the planar vector field $\grad^f=\de_\eta+f\de_\tau$.
We will thus take variations of $f$ via smooth one-parameter families of diffeomorphisms $\phi_\epsilon:\R^2\to\R^2$, i.e., by smoothly varying the integral curves of $\grad^f$; see Section~\ref{sec11251132}.
We will then prove that this approach is equivalent to the use of contact diffeomorphisms $\Phi_\epsilon:\HH\to\HH$; see Section~\ref{sec11241822}.

Finally, we will consider functions $f\in\Cow$ that are solutions to the equation $\grad^f\grad^ff=0$ in a Lagrangian sense, that is, functions such that $\grad^ff$ is constant along the integral curves of $\grad^f$.
We will characterize such functions as the ones for which the integral curves of $\grad^f$ are parabolas, or, equivalentely, as the ones whose graph $\Gamma_f$ is ruled by horizontal straight lines.
These functions are the ones appearing in Theorem \ref{teo10261843}.\\

The paper is organized as follows.
Section~\ref{sec10262024} is devoted to the presentation of all main definitions.
In the next Section~\ref{sec10262025}, we study solutions to the equation $\grad^f\grad^ff=0$.
We construct a Lagrangian variation of a function $f\in\Cow$ in Section~\ref{sec11251132}.
In Section~\ref{sec11241822}, we prove some basic properties of contact diffeomorphisms.
Section~\ref{sec10262029} is devoted to the first contact variation and Section~\ref{sec10262031} to the second contact variation for functions $f\in\Cow$.
Finally, in Section~\ref{sec10262032} we prove our main theorem.
An Appendix is added as a reference for a few equalities that are applied all over the paper.\\

The author thanks his advisor Francesco Serra Cassano for many fruitful discussions and Katrin Fässler for her comments and corrections on a draft of this paper.

\section{Preliminaries}\label{sec10262024}
\subsection{The Heisenberg group}\label{subs11191130}
The \emph{first Heisenberg group} $\HH$ is the connected, simply connected Lie group associated to the Heisenberg Lie algebra $\hh$.
The \emph{Heisenberg Lie algebra $\hh$} is the only three-dimensional nilpotent Lie algebra that is not commutative.
It can be proven that, for any two linearly independent vectors $A,B\in\hh\setminus[\hh,\hh]$, the triple $(A,B,[A,B])$ is a basis of $\hh$ and $[A,[A,B]]=[B,[A,B]]=0$.
The Heisenberg group has the structure of a \emph{stratified Lie group}, i.e., $\hh = \Span\{A,B\} \oplus \Span\{[A,B]\}$, see \cite{2015arXiv150903881L,2016arXiv160408579L}.

We then identify $\HH = (\Span\{A,B,[A,B]\},*)$, where
\[
p*q := p+q+\frac12[p,q].
\]
In the coordinates $(x,y,z) = xA+yB+z[A,B]$, which are the exponential coordinates of first kind, we have
\[
(a,b,c)*(x,y,z) = (a+x , b+y , c+z+\frac12(ay-bx) ) .
\]
The inverse is $(x,y,z)^{-1}=(-x,-y,-z)$.

The elements $A,B,[A,B]\in \hh$ induce a frame of left-invariant vector fields on $\HH$:
\[
X:=\de_x - \frac12 y\de_z,
\qquad
Y:=\de_y + \frac12 x\de_z,
\qquad
Z:=\de_z .
\]
The \emph{horizontal subbundle} is the vector bundle 
\[
H:=\bigsqcup_{p\in\HH} \Span\{X(p),Y(p)\}\subset T\HH .
\]

The maps $\delta_\lambda(x,y,z):=(\lambda x,\lambda y,\lambda^2 z)$, $\lambda>0$, are called \emph{dilations}. 
They are group automorphisms of $\HH$ and for all $\lambda,\mu>0$ it holds $\delta_\lambda\circ\delta_\mu=\delta_{\lambda\mu}$.

\subsection{Intrinsic graphs and intrinsic differentials}\label{subs11241752}

A \emph{vertical plane} is 
a plane containing the $z$-axis.
Explicitly, for $\theta\in\R$,
\[
\bb W_\theta := \left\{ (\eta\sin\theta , \eta\cos\theta , \tau):\eta,\tau\in\R \right\} \subset\HH .
\]
Vertical planes are the only 2-dimensional subgroups of $\HH$ that are $\delta_\lambda$-homogeneous, i.e., $\delta_\lambda(\bb W_\theta)=\bb W_\theta$ for all $\lambda>0$.

The \emph{intrinsic $X$-graph} (or simply \emph{intrinsic graph}) of a function $f:\R^2\to\R$ is the set\footnote{
In a different choice of coordinates in $\HH$, we can have $(0,\eta,\tau)*(f(\eta,\tau),0,0)= (f(\eta,\tau),\eta,\tau)$. For instance, we will use these coordinates in Section~\ref{sec10261906}}
\begin{align*}
\Gamma_f &:= \left\{ (0,\eta,\tau)*(f(\eta,\tau),0,0):\eta,\tau\in\R^2 \right\} \\
	&= \left\{ (f(\eta,\tau),\eta,\tau - \frac12\eta f(\eta,\tau)) : \eta,\tau\in\R^2 \right\} .
\end{align*}
If one look at $f$ as a function $\bb W_{0}\to\Span\{A\}$, then $\Gamma_f = \{p*f(p):p\in\bb W_{0}\}$.
Left translations and dilations of an intrinsic graph are also intrinsic graphs.
For $\alpha\in\R$, the vertical plane $\bb W_{\arctan(\alpha)}$ is the intrinsic graph of the function $f(\eta,\tau)=\alpha\eta$.
We will use the map $\pi_{X}:\HH\to\R^2$, $\pi_{X}(x,y,z) = (y,z+\frac12 xy)$.
Note that $\pi_{X}(p*f(p)) = p$.

For $(\eta_0,\tau_0)\in\R^2$ and $f:\R^2\to\R$ continuous, set $f_0 := f(\eta_0,\tau_0)$ and $p_0 := (0,\eta_0,\tau_0)*(f_0,0,0) = (f_0, \eta_0, \tau_0 -\frac12\eta_0 f_0)$.
We say that $f$ is \emph{intrinsically $\Co^1$}, or \emph{belonging to ${\Cow}$}, with differential $\psi:\R^2\to\R$, if 
$\delta_\lambda(p_0^{-1}\Gamma_f)$ converge to $\bb W_{\arctan(\psi(\eta_0,\tau_0))}$ in the sense of the local Hausdorff convergence of sets as $\lambda\to\infty$, and the convergence is uniform on compact sets in $(\eta_0,\tau_0)$.

Notice that $\delta_\lambda(p_0^{-1}\Gamma_f) = \Gamma_{f_{(\eta_0,\tau_0);\lambda} }$, where
\[
f_{(\eta_0,\tau_0);\lambda} ( \eta, \tau) 
= \lambda\left( -f_0 + f(\eta_0 + \frac{ \eta}\lambda , \tau_0 + f_0\frac{ \eta}{\lambda} + \frac{ \tau}{\lambda^2}  ) \right) .
\]
Therefore, $f$ belongs to ${\Cow}$ 
with differential $\psi$
if and only if $f_{(\eta_0,\tau_0);\lambda}$ converge uniformly on compact sets to the function $(\eta,\tau)\mapsto\psi(\eta_0,\tau_0)\eta$, as $\lambda\to+\infty$, and the convergence is uniform on compact sets in $(\eta_0,\tau_0)$.
Notice that $\psi$ has to be continuous.

The \emph{intrinsic gradient} of a function $f:\R^2\to\R$ is the vector field on $\R^2$ defined as
\[
\grad^f := \de_\eta + f \de_\tau .
\]
We can express the intrinsic differentiability in terms of the differentiability of $f$ along the integral curves of $\grad^f$:
from \cite[Theorem 4.95]{MR3587666} we obtain the following characterisation, which justify the notation $\grad^ff$ for the differential $\psi$ of $f\in{\Cow}$.
\begin{lemma}\label{lem11190949}
	A continuous function $f:\R^2\to \R$ is in ${\Cow}$ with differential $\psi$ if and only if for every $p\in\R^2$ there exists a $\Co^2$-function $g_p:I\to\R$, where $I\subset\R$ is a neighbourhood of $0$, such that 
	\[
	\begin{cases}
	g_p(0) = 0, \\
	g_p'(t) = f(p+(t,g_p(t))) &\forall t\in I , \\
	g_p''(t) = \psi(p+(t,g_p(t))) &\forall t\in I .
	\end{cases}
	\]
\end{lemma}
Note that $t\mapsto p+(t,g_p(t))$ is an integral curve of $\grad^f$ and that $g_p$ is not unique in general.
Another interpretation of these curves will be useful:
\begin{lemma}\label{lem11191340}
	Let $f\in{\Cow}$.
	A  curve $\gamma:I\to\R^2$ of class $\Co^1$, where $I\subset\R$ is an interval, is an integral curve of $\grad^f$ if and only if the curve $t\mapsto\gamma(t)*f(\gamma(t))\in\Gamma_f$ is a curve of class $\Co^1$ tangent to the horizontal bundle $H$.
\end{lemma}
\begin{remark}\label{rem10261128}
	In \cite{MR2124590} it has been shown that there exists $f\in\Cow$ whose intrinsic graph $\Gamma_f$ has Euclidean Hausdorff dimension (seen as a subset of the Euclidean $\R^3$) strictly larger than two.
	It is possible to prove, for example using Lemma~\ref{lem10261002}, that $\Gamma_{f+1}$ does not have locally finite intrinsic perimeter and in particular $f+1\notin\Cow$.
	This shows that $\Cow$ is not a vector space.
\end{remark}

\subsection{Smooth approximation}

A sequence $\{f_k\}_{k\in\N}\subset\Cow$ converges to $f$ in $\Cow$ if $f_k$ and $\grad^{f_k}f_k$ converge to $f$ and $\grad^ff$ uniformly on compact sets.
The following lemma has been proven in \cite{MR2223801}.
\begin{lemma}\label{lem10241912}
	If $f\in\Cow$ 
	then there is a sequence of functions $\{f_k\}_{k\in\N}\subset\Co^\infty(\R^2)$ 
	that converges to $f$ in $\Cow$.
\end{lemma}

\subsection{Perimeter and Bernstein's Problem}\label{subs11241814}
The Lebesgue measure $\scr L^3$ in $\R^3$ is a Haar measure on $\HH$ in the exponential coordinates introduced Section \ref{subs11191130}. 
Notice that for any measurable set $E\subset\HH^1$ and any $\lambda>0$ it holds $\scr L^3(\delta_\lambda(E)) = \lambda^4 \scr L^3(E)$.

Let $\langle\cdot,\cdot \rangle$ be the left-invariant scalar product on the subbundle $H$ such that $(X,Y)$ is an orthonormal frame and set $\|v\| := \sqrt{\langle v,v \rangle}$ for $v\in H$.
The \emph{sub-Riemannian perimeter} of a measurable set $E\subset\HH^1$ in an open set $\Omega$ is
\[
P(E;\Omega) := \sup\left\{
\int_E \div V \dd\scr L^3:\ V\in\Gamma(H),\ {\rm spt}(V)\ssubset \Omega,\ \|V\|\le 1
\right\},
\]
where $\Gamma(H)$ contains all the smooth sections of the horizontal subbundle and 
$\div V$ is the divergence of vector fields on $\R^3$.
One can show that, for every $V_1,V_2\in\Co^\infty(\R^3)$,
\[
\div(V_1 X+V_2 Y) = XV_1 + YV_2 .
\]
A set $E$ has \emph{locally finite perimeter} if $P(E;\Omega)<\infty$ for all $\Omega\subset\HH$ open and bounded.
If $E$ has locally finite perimeter, the function $\Omega\mapsto P(E;\Omega)$ induces a Radon measure $|\de E|$ on $\HH^1$, which is concentrated on the so-called \emph{reduced boundary} $\de^*E\subset \de E$.
Moreover, up to a set of $|\de E|$-measure zero and a rotation around the $z$-axis, $\de^*E$ is the countable union of intrinsic graphs of $\Cow$ functions. 
See \cite{MR1871966} and \cite{MR2313532} for further reading.

A measurable set $E$ has \emph{minimal perimeter} if, for every bounded open set $\Omega\subset\HH^1$ and every measurable set $F\subset\HH^1$ with symmetric difference $E\Delta F\ssubset \Omega$, we have
\[
P(E;\Omega) \le P(F;\Omega) .
\]
In this case, the reduced boundary $\de^*E$ of $E$ is called \emph{area-minimizing surface}.
We are interested in area minimizers that are global intrinsic graphs.
\begin{Coj}[Bernstein's Problem]
	If $f\in{\Cow}$ is such that $\Gamma_f$ is an area-minimizing surface, then $\Gamma_f$ is a vertical plane up to left-translations.
\end{Coj}
Such conjecture has been proven in the case $f\in\Co^1(\R^2)$ in \cite{MR3406514}, while it has been presented a counterexample in \cite{MR2455341} with $f\in\Co^0(\R^2)\setminus {\Cow}$.

For an open domain $\omega\subset\R^2$, set 
\[
\Omega_\omega:=\{(0,\eta,\tau)*(\xi,0,0):(\eta,\tau)\in\omega,\ \xi\in\R\}.
\]
If $f\in{\Cow}$ and $E_f=\{(0,\eta,\tau)*(\xi,0,0)\in\R^2,\ \xi\le f(\eta,\tau)\}$, then
\[
P(E_f;\Omega_\omega) = \int_\omega \sqrt{1+(\grad^ff)^2} \dd\eta\dd\tau .
\]
If $E_f$ has minimal perimeter, then, for every $g\in{\Cow}$ with $\{f\neq g\}\ssubset \omega$, it holds
\begin{equation*}
\int_\omega \sqrt{1+(\grad^ff)^2} \dd\eta\dd\tau \le \int_\omega \sqrt{1+(\grad^gg)^2} \dd\eta\dd\tau .
\end{equation*}
It is not known whether the converse implication holds.

\section{Lagrangian solutions to $\Delta^ff=0$}\label{sec10262025}
For $f\in\Cow$ and $v\in\Co^2(\R^2)$, we define the differential operator
\begin{equation}\label{eq10261049}
\Delta^fv := \de_\eta^2v+2f\de_\eta\de_\tau v + f^2\de_\tau^2v + \grad^ff \de_\tau v .
\end{equation}
Notice that, if $f\in\Co^2(\R^2)$, then
\begin{equation*}
 	\Delta^fv = \grad^f(\grad^fv) .
\end{equation*}
The next lemma will be a fundamental tool for extending some results beyond the smooth case via approximation. 
The proof trivially follows from the explicit expressions of the differential operators $\grad^f$ and $\Delta^f$.
\begin{lemma}\label{lem10241748}
	If $\{f_k\}_{k\in\N}\subset\Cow$ and $\{v_k\}_{k\in\N}\subset\Co^2(\R^2)$ are sequences converging to $f$ and $v$ in their respective spaces, then the sequences $\{\grad^{f_k}v_k\}_{k\in\N}$ and $\{\Delta^{f_k}v_k\}_{k\in\N}$ converge to $\grad^fv$ and $\Delta^fv$ uniformly on compact sets.
\end{lemma}

If $f\in\Co^2(\R^2)$ is such that $\Gamma_f$ is a minimal surface in $\HH$, then one  shows that $f$ satisfies the differential equation (see \cite{MR2333095})
\begin{equation}\label{eq10241555}
 \grad^f\left( \frac{ \grad^ff }{ \sqrt{1+(\grad^ff)^2} }\right) = 0 .
\end{equation}
Equation \eqref{eq10241555} is equivalent, for $f\in\Co^2(\R^2)$, to 
\begin{equation}\label{eq10241526}
 	\Delta^ff=0 .
\end{equation}
For a generic $f\in\Cow$, equation \eqref{eq10241526} has not the classical interpretation~\eqref{eq10261049}.
However, using a ``Lagrangian interpretation'' of $\grad^f(\grad^ff)=0$, we give the following definition:
\begin{definition}\label{def08101850}
	A function $f\in\Cow$ satisfies \emph{$\Delta^ff=0$ in weak Lagrangian sense}, if for every $p\in\R^2$ there is an integral curve $\gamma$ of $\grad^f$ passing through $p$ such that $ \grad^ff$ is constant along $\gamma$.
\end{definition}
If $f\in\Co^2(\R^2)$ then $\Delta^ff = \grad^f(\grad^ff)=0$ if and only if $\grad^ff$ is constant along \emph{all} integral curves of $\grad^f$, i.e., $\Delta^ff=0$ holds in a \emph{strong} Lagrangian sense, see Remark~\ref{rem08140839}.

Lemma~\ref{lem11191504} will characterize such functions by the integral curves of $\grad^f$. 

\begin{lemma}\label{lem10251359}
	Let $A,B\in\Co^0(\R)$.
	The map $\R^2\to\R^2$ given by
	\[
	\Psi:(t,\zeta)\mapsto \left(t,\frac{A(\zeta)}{2} t^2 + B(\zeta) t + \zeta \right)
	\]
	is a homeomorphism if and only if
	\begin{enumerate}
	\item
	For all $\zeta,\zeta'\in\R$ 
		\begin{enumerate}
		\item[(1a)] 	either $A(\zeta)=A(\zeta')$ and $B(\zeta) = B(\zeta')$,
		\item[(1b)] 	or $2 \big(A(\zeta)-A(\zeta')\big) (\zeta-\zeta') > \big(B(\zeta)-B(\zeta')\big)^2 $.
		\end{enumerate}
	\item If there exists $\zeta_0\in\R$ such that $A(\zeta_0)>0$, then 
		\[
		\limsup_{\zeta\to\infty} \left(\zeta-\frac{B(\zeta)^2}{2A(\zeta)}\right) = +\infty .
		\]
	\item If there exists $\zeta_0\in\R$ such that $A(\zeta_0)<0$, then 
		\[
		\liminf_{\zeta\to-\infty} \left(\zeta-\frac{B(\zeta)^2}{2A(\zeta)}\right) = -\infty .
		\]
	\end{enumerate}
\end{lemma}
\begin{proof}
	Define $ g (t,\zeta) = \frac{A(\zeta)}{2} t^2 + B(\zeta) t + \zeta$, so that $\Psi(t,\zeta)=\left(t, g (t,\zeta)\right)$.
	We first show that $\Psi$ is injective if and only if property $(1)$ holds.
	For $\zeta'>\zeta$, define the quadratic polynomial
	\[
	Q_{\zeta',\zeta}(t)= g (t,\zeta')- g (t,\zeta)
	= \frac{A(\zeta')-A(\zeta)}{2} t^2 + (B(\zeta')-B(\zeta)) t + (\zeta'-\zeta) .
	\]
	The map $\Psi$ is injective if and only if for all $\zeta',\zeta\in\R$ with $\zeta'>\zeta$ the polynomial $Q_{\zeta',\zeta}$ has no zeros.
	If $A(\zeta')=A(\zeta)$, then $Q_{\zeta',\zeta}$ is in fact linear, thus it has no zeros if and only if $B(\zeta')=B(\zeta)$ and we obtain property~(1a).
	If $A(\zeta')\neq A(\zeta)$, then $Q_{\zeta',\zeta}$ has no zeros if and only if its discriminant is strictly negative, i.e., property~(1b) holds.
	
	Next, we assume that $\Psi$ is injective, i.e., that property (1) holds, and we will show that $\Psi$ is surjective if and only if properties (2) and (3) hold.
	By the Invariance of Domain Theorem, the fact that $\Psi$ is surjective is equivalent to $\Psi$ being a homeomorphism.
	Notice that, since $Q_{\zeta',\zeta}(0)=\zeta'-\zeta>0$ for all $\zeta'>\zeta$, we have
	\begin{equation}\label{eq07172015}
	\zeta'>\zeta 
	\quad\THEN\quad
	\forall t\in\R\quad g (t,\zeta')> g (t,\zeta) .
	\end{equation}
	
	Suppose that $\Psi$ is surjective, hence a homeomorphism.
	Suppose $\zeta_0\in\R$ is such that $A(\zeta_0)>0$.
	By~$(1)$ we have that $A$ is monotone increasing, therefore $A(\zeta)>0$ for all $\zeta\ge\zeta_0$.
	It follows that if $\zeta\ge\zeta_0$ then
	\[
	\zeta-\frac{B(\zeta)^2}{2A(\zeta)} = \inf_{t\in\R}  g (t,\zeta).
	\]
	For $M\in\R$ define $K_M=\{(\eta,\tau)\in\R^2: g (\eta,\zeta_0)\le \tau\le M\}$.
	Since $A(\zeta_0)>0$, the set $K_M$ is compact (possibly empty) for all $M\in\R$.
	Next, for $\zeta\in\R$ define $U_\zeta=\Psi(\R\times(-\infty,\zeta))=\{(\eta,\tau):\tau<g(\eta,\zeta)\}$. 
	Since $\Psi$ is surjective, the open sets $U_\zeta$ cover $\R^2$.
	Hence, there is $\zeta_1\ge\zeta_0$ such that $K_M\subset U_{\zeta_1}$.
	Using~\eqref{eq07172015}, we obtain
	\[
	\forall \zeta\ge\zeta_1\quad
	\inf_{t\in\R}  g (t,\zeta) \ge M.
	\]
	Since $M$ is arbitrary, we have proven~(2).
	Property~(3) is proven with a similar argument.
	
	Now we prove the converse implication. 
	Suppose that $A$ and $B$ satisfy properties (2) and (3). 
	In order to prove that $\Psi$ is surjective, we need only to prove that $\lim_{\zeta\to\infty} g (t,\zeta)=+\infty$ and $\lim_{\zeta\to-\infty} g (t,\zeta)=-\infty$, for every $t\in\R$.
	
	If $A(\zeta)=A(0)$ for all $\zeta\ge0$, then $ g (t,\zeta)= g (t,0)+\zeta$ and thus $\lim_{\zeta\to\infty} g (t,\zeta)=+\infty$.
	If $A(\zeta)\le0$ for all $\zeta\in\R$, then there is $C>0$ such that $0\le A(\zeta)-A(0) \le C$ for all $\zeta>0$.
	We may suppose $A(\zeta)>A(0)$ for $\zeta$ large enough.
	Thus, using $(1b)$, 
	\begin{align*}
	 g (t,\zeta) &\ge \frac{A(0)}{2} t^2 + B(0) t + \zeta + (B(\zeta)-B(0)) t \\
	&\ge  \frac{A(0)}{2} t^2 + B(0) t + \zeta - |t| \sqrt{2(A(\zeta)-A(0))\zeta}  \\
	&\ge  \frac{A(0)}{2} t^2 + B(0) t + \zeta - |t|\sqrt{2C} \sqrt{\zeta}  .
	\end{align*}
	The limit $\lim_{\zeta\to\infty} g (t,\zeta)=+\infty$ follows.
	Finally, if $A(\zeta_0)>0$ for some $\zeta_0\in\R$, then for all $\zeta\ge\zeta_0$ we have
	$\inf_{t\in\R} g (t,\zeta) = \zeta-\frac{B(\zeta)^2}{2A(\zeta)}$.
	Property (2) implies  that $\lim_{\zeta\to\infty} g (t,\zeta)=+\infty$.
	
	The limit $\lim_{\zeta\to-\infty} g (t,\zeta)=-\infty$ is deduced similarly from (3).
\end{proof}
\begin{remark}\label{rem08101832}
	If $A,B\in\Co(\R)$ satisfy properties (1), (2) and (3) of the previous Lemma~\ref{lem10251359}, then the function $f$ defined by $f(\Psi(t,\zeta)) = \de_tg(t,\zeta) = A(\zeta) t + B(\zeta)$ belongs to $\Cow$ by Lemma~\ref{lem11190949}.
	Moreover, the curves $t\mapsto g(t,\zeta)$ are integral curves of $\grad^f$ along which $\grad^ff(\Psi(t,\zeta)) = \de_t^2g(t,\zeta) = A(\zeta)$ is constant.
	So, $\Delta^ff=0$ in weak Lagrangian sense.
	The graphs of these functions are examples of ``graphical strips'' as introduced in \cite{MR2472175}.
	For example, for any $A\in\Co^0(\R)$ non-decreasing, we can define $g(t,\zeta):=A(\zeta)t^2+\zeta$ and we obtain a well defined $f\in{\Cow}$ with $\Delta^ff=0$ given by
	\[
	f(t,A(\zeta)t^2+\zeta) = 2A(\zeta)t .
	\]
	The converse also holds, as the next lemma shows.
\end{remark}
\begin{lemma}\label{lem11191504}
	Let $f\in{\Cow}$ satisfying $\Delta^ff=0$ in weak Lagrangian sense.
	Then the curves $t\mapsto(t,g(t,\zeta))$, where $\zeta\in\R$ and
	\begin{equation}\label{eq08101341}
	g(t,\zeta) = \frac{\grad^ff(0,\zeta)}{2} t^2 + f(0,\zeta) t + \zeta ,
	\end{equation}
	are the integral curves of $\grad^f$ along which $\grad^ff$ is constant.
	Moreover, the functions $\zeta\mapsto \grad^ff(0,\zeta)$ and $\zeta\mapsto f(0,\zeta)$ satisfy the conditions (1), (2) and (3) in Lemma~\ref{lem10251359}.
	In particular, $\tau\mapsto \grad^ff(\eta,\tau)$ is non-decreasing, for all $\eta\in\R$.
\end{lemma}
\begin{proof}[Proof of Lemma \ref{lem11191504}]
	Given a function $g_p:I\to\R$ like in Lemma \ref{lem11190949}  along which $\grad^ff$ is constant,
	we have $g_p''(t)=\grad^ff(p)$ for all $t\in I$, i.e., $g_p$ is a polynomial of second degree.
	Moreover, such a $g_p$ is unique for every $p$, because it is completely determined by $f(p)$ and $\grad^ff(p)$.
	
	It follows that $g_p$ is defined on $\R$.
	Indeed, suppose $I=(a,b)$ and set $q=\lim_{t\to b}p+(t,g_p(t))$, which exists because $g_p$ is a polynomial.
	If $g_q:J\to\R$ is a function like in Lemma \ref{lem11190949} along which $\grad^ff$ is constant, then $g_q$ is uniquely determined by $f(q)$ and $\grad^ff(q)$, where
	\begin{align*}
	f(q) &= \lim_{t\to b}f(p+(t,g_p(t))) = \lim_{t\to b} g_p'(t) , \\
	\grad^ff(q) &= \lim_{t\to b}\grad^ff(p+(t,g_p(t))) = \lim_{t\to b}g_p''(t) .
	\end{align*}
	Hence, $g_q(t)=g_p(b+t)$ for $t<0$ and so $g_p$ can be extended beyond $b$.
	Similarly, we can extend $g_p$ to values below $a$.
	
	If we consider $p=(0,\zeta)$, then $g_p(t)=g(t,\zeta)$, where $g(t,\zeta)$ is given in \eqref{eq08101341}.
	If $p\in\R^2$, then the curve $t\mapsto p+(t,g_p(t))$ intersects the axis $\{0\}\times\R$ at some point, and thus $g_p$ is of the form described in \eqref{eq08101341} up to a change of variables in $t$.
	We conclude that 
	the map $(t,\zeta)\mapsto (t,g(t,\zeta))$ is a homeomorphism.
	Therefore, the conditions stated in Lemma~\ref{lem10251359} hold true.
	
	Finally, since $(f(0,\zeta)-f(0,\zeta'))^2\ge0$, then $\zeta\mapsto \grad^ff(0,\zeta)$ is non-decreasing.
	Since $\grad^ff(t,g(t,\zeta)) = \grad^ff(0,\zeta)$ and since, for $t\in\R$ fixed, the map $\zeta\mapsto g(t,\zeta)$ is a ordering-preserving homeomorphism $\R\to\R$, then the map $\tau\mapsto\grad^ff(\eta,\tau)$ is non-decreasing as well, for all $\eta\in\R$.
\end{proof}
\begin{remark}
	Lemma~\ref{lem11191504} states in particular that, if $\Delta^ff=0$ in weak Lagrangian sense then $\Gamma_f$ is foliated by horizontal straight lines.
	Indeed, notice that any parabola $t\mapsto g(t,\zeta)$ in $\R^2$ lifts to a straight line in $\Gamma_f$. 
	In \cite[Theorem 3.5]{MR3406514} Galli and Ritoré are able to prove that, if $f\in\Co^1(\R^2)$ and if $\Gamma_f$ is a minimal surface in $\HH$, then $\Gamma_f$ is foliated by horizontal straight lines, i.e., $\Delta^ff=0$ holds in weak Lagrangian sense. 
\end{remark}
\begin{remark}\label{rem08140839}
	One may wonder wether Definition~\ref{def08101850} for weak Lagrangian solutions to $\Delta^ff=0$ is equivalent to a stronger condition, namely that $\grad^ff$ is constant along \emph{all} integral curves of $\grad^f$.
	This is the case when $f\in C^1(\R^2)$, because integral curves are unique at each point.
	The following example shows that strong and weak conditions are not equivalent.
	Indeed, there are functions for which the curves $t\mapsto (t,g(t,\zeta))$ described in Lemma~\ref{lem11191504} do not exhaust all the integral curves of $\grad^f$.
\end{remark}
Let $h\in\Co^2(\R)$ and define $k:\R^2\to\R$ by requiring that for each $s\in\R$ the function $t\mapsto k(t,s)$ is the unique polynomial of second degree with $k(s,s)=h(s)$, $\de_tk(s,s)=h'(s)$ and $\de^2_tk(s,s)=h''(s)$.
Explicitly, we have
\[
k(t,s) = \frac{h''(s)}{2} t^2 + (h'(s)-h''(s)s) t + h(s)-h'(s)s+\frac{h''(s)}2s^2 .
\]
If the map $K(t,s)=(t,k(t,s))$ is a homeomorphism $\R^2\to\R^2$, then we may define a function $f\in\Cow$ by $f(K(t,s))=\de_tk(t,s)$, 
as we did in Remark~\ref{rem08101832}.
Then $t\mapsto K(t,s)$ are integral curves of $\grad^f$ and $\grad^ff(K(t,s))=\de_t^2k(t,s)=h''(s)$.
It follows that $\Delta^ff=0$ holds in weak Lagrangian sense.
However, $s\mapsto K(s,s)=(s,h(s))$ is an integral curve of $\grad^f$, because $f(K(s,s))=h'(s)$.
Since $\grad^ff(K(s,s))=h''(s)$, there is no need for $\grad^ff$ to be constant along this curve.

As an example, consider $h(s)=s^3$, for which we have $k(t,s)=3st^2-3s^2t+s^3$.
We show that the map $K$ is in this case a homeomorphism.
Define $\zeta(s)=s^3$, $A(\zeta(s))=6s=6\zeta^{1/3}$, $B(\zeta(s))=-3s^2=-3\zeta^{2/3}$ and the functions $g(t,\zeta)$ and $G(t,\zeta)$ as in Lemma~\ref{lem10251359}. 
Since $K(t,s)=G(t,\zeta(s))$ and since $\zeta(\cdot)$ is a homeomorphism $\R\to\R$, we need only to show that $G$ is a homeomorphism $\R^2\to\R^2$, i.e., that the functions $A$ and $B$ satisfy all conditions of Lemma~\ref{lem10251359}:
\begin{enumerate}
\item[(1):]
	Let $\zeta=\zeta(s),\zeta'=\zeta(s')\in\R$.
	If $A(\zeta)=A(\zeta')$, then $s=s'$ and thus $B(\zeta)=B(\zeta')$.
	If instead $A(\zeta)\neq A(\zeta')$, then $s\neq s'$ and thus
	\begin{multline*}
	2(A(\zeta)-A(\zeta'))(\zeta-\zeta') - (B(\zeta)-B(\zeta'))^2 \\
	= 2 (6s-6s')(s^3-s'^3) - 9 (s'^2-s^2)^2
	= 3 (s-s')^4 >0 .
	\end{multline*}
\item[(2)\&(3):]
	Since $\zeta-\frac{B(\zeta)^2}{2A(\zeta)} = \frac14 s^3$ and since $\zeta\to\pm\infty$ if and only if $s\to\pm\infty$, then $\lim_{\zeta\to+\infty} \zeta-\frac{B(\zeta)^2}{2A(\zeta)} = \lim_{s\to+\infty}\frac14 s^3 = +\infty$ and $\lim_{\zeta\to-\infty} \zeta-\frac{B(\zeta)^2}{2A(\zeta)} = \lim_{s\to-\infty}\frac14 s^3 = -\infty$.
\end{enumerate}
The function $f$ can be explicitly computed as 
$f(\eta,\tau) = 3\eta^2 - 3(\tau-\eta^3)^{2/3} $.
Finally, as we noticed before, $s\mapsto (s,s^3)$ is an integral curve of $\grad^f$ and $\grad^ff(s,s^3)=6s$ is not constant.


\begin{lemma}\label{lem10251843}
	Let $f\in\Cow$. 
	If $\Delta^ff=0$ in weak Lagrangian sense, then there is a sequence $\{f_k\}_{k\in\N}\subset\Co^\infty(\R^2)$ converging to $f$ in $\Cow$  such that $\Delta^{f_k}f_k=0$ for all $k\in\N$.
\end{lemma}
\begin{proof}
	Let $\{\rho_\epsilon\}_{\epsilon>0}\subset\Co^\infty(\R)$ be a family of mollifiers with 
	$\spt(\rho_\epsilon)\subset[-\epsilon,\epsilon]$, $\rho_\epsilon\ge0$, $\rho_\epsilon(0)>0$ and $\int_{\R}\rho_\epsilon(r)\dd r = 1$.
	Fix $f\in\Cow$ with $\Delta^ff=0$.
	Set $A(\zeta) := \grad^ff(0,\zeta)$ and $B(\zeta):=f(0,\zeta)$.
	Define
	\begin{align*}
	 	A_\epsilon(\zeta) &:= \int_{\R} \grad^ff(0,\zeta-r)\rho_\epsilon(r)\dd r , \\
		B_\epsilon(\zeta) &:= \int_{\R} f(0,\zeta-r)\rho_\epsilon(r)\dd r , \\
		g_\epsilon(t,\zeta) &:= \frac{A_\epsilon(\zeta)}{2} t^2 + B_\epsilon(\zeta) t + \zeta .
	\end{align*}
	We claim that, for all $\epsilon>0$, all conditions stated in Lemma~\ref{lem10251359} hold for $A_\epsilon$ and $B_\epsilon$.
	Let $\zeta,\zeta'\in\R$ with $\zeta<\zeta'$.
	First, suppose that $A_\epsilon(\zeta)=A_\epsilon(\zeta')$.
	Notice that $A(\zeta-r)-A(\zeta'-r)\le 0$ for all $r\in\R$, because $A$ is non-decreasing.
	Thus, we deduce from
	\[
	0=A_\epsilon(\zeta)-A_\epsilon(\zeta')
	= \int_\R (A(\zeta-r)-A(\zeta'-r)) \rho_\epsilon(r) \dd r
	\]
	that $(B(\zeta-r)-B(\zeta'-r))\rho_\epsilon(r)=0$ for all $r\in\R$ and therefore that $B_\epsilon(\zeta)=B_\epsilon(\zeta')$, i.e., (1a) holds\\
	Second, suppose that $A_\epsilon(\zeta)\neq A_\epsilon(\zeta')$.
	Using Jensen's inequality, we have
	\begin{multline*}
	2 \big(A_\epsilon(\zeta)-A_\epsilon(\zeta')\big)(\zeta-\zeta') \\
	= \int_\R 2\big(A(\zeta-r)-A(\zeta'-r)\big) \big((\zeta-r)-(\zeta'-r)\big) \rho_\epsilon(r) \dd r \\
	> \int_\R \big(B(\zeta-r)-B(\zeta'-r)\big)^2\rho_\epsilon(r)\dd r 
	\ge \left( \int_\R (B(\zeta-r)-B(\zeta'-r))\rho_\epsilon(r) \dd r \right)^2.
	\end{multline*}
	So, condition (1b) is also verified.\\
	Suppose that $A_\epsilon(\zeta_0)>0$ for some $\zeta_0\in\R$.
	By the monotonicity of $A$ and the positivity of $\rho_\epsilon$, we may assume $A(\zeta_0)>0$. 
	Let $M>0$.
	Since property $(2)$ of Lemma~\ref{lem10251359} holds for $A$,
	there is $\zeta_1>\zeta_0$ so that for all $\zeta>\zeta_1$
	\[
	M < \zeta - \frac{B(\zeta)^2}{2A(\zeta)} = \frac{2A(\zeta)\zeta-B(\zeta)^2}{2A(\zeta)} .
	\]
	Using Jensen inequality, we have for all $\zeta>\zeta_1+\epsilon$
	\begin{multline*}
	2 A_\epsilon(\zeta)\zeta-B_\epsilon(\zeta)^2
	\ge \int_\R \left( 2\zeta A(\zeta-r) - B(\zeta-r)^2 \right) \rho_\epsilon(r) \dd r \\
	= 2 \int_\R A(\zeta-r) r\rho_\epsilon(r) \dd r + \int_\R \left(2(\zeta-r)A(\zeta-r) - B(\zeta-r)^2\right) \rho_\epsilon(r) \dd r \\
	\ge -2\epsilon \int_\R A(\zeta-r) \rho_\epsilon(r) \dd r +  2M \int_\R A(\zeta-r) \rho_\epsilon(r) \dd r 
	= 2A_\epsilon(\zeta) (M-\epsilon).
	\end{multline*}
	Thus, $M-\epsilon < \zeta - \frac{B_\epsilon(\zeta)^2}{2A_\epsilon(\zeta)}$ for all $\zeta>\zeta_1+\epsilon$.
	Since $M$ was arbitrary, we obtain property~$(2)$ of Lemma~\ref{lem10251359}.
	Property~$(3)$ can be similarly obtained.
	
	 The functions $G_\epsilon:\R^2\to\R^2$, $G_\epsilon(t,\zeta) := (t,g_\epsilon(t,\zeta))$, 
	are homeomorphisms and, as $\epsilon\to0$, they converge to $G_0$ uniformly on compact sets.
	It follows that $G_\epsilon^{-1}$ also converge to $G_0^{-1}$, as $\epsilon\to0$.
	
	 For $\epsilon>0$, define $f_\epsilon\in\Co^\infty(\R^2)$ via
	\[
	f_\epsilon(t,g_\epsilon(t,\zeta)) = A_\epsilon(\zeta)t+B_\epsilon(\zeta).
	\]
	By the continuity of $G_\epsilon$ and $G_\epsilon^{-1}$ in $\epsilon$, $f_\epsilon$ and $\grad^{f_\epsilon}f_\epsilon$ converge to $f_0$ and $\grad^{f_0}{f_0}$ uniformly on compact sets. 
	Finally, $\Delta^{f_\epsilon}f_\epsilon=0$ by construction.
\end{proof}

\section{A Lagrangian approach to contact variations}\label{sec11251132}

\begin{proposition}\label{prop09141457}
	Let $\phi=(\phi_1,\phi_2):\R^2\to\R^2$ be a $\Co^\infty$-diffeomorphism.
	Let $f\in\Cow$ and assume
	\begin{equation}\label{eq07021739}
	\grad^f\phi_1(p)\neq0\qquad\forall p\in\R^2 .
	\end{equation}
	Define $\bar f:\R^2\to\R$ as
	\[
	\bar f\circ\phi = \frac{\grad^f\phi_2}{\grad^f\phi_1} .
	\]
	Then $\bar f\in\Cow$ and
	\begin{equation}\label{eq10241705}
	\grad^{\bar f}\bar f\circ\phi 
	= \frac{\Delta^f\phi_2}{(\grad^f\phi_1)^2} - \frac{\grad^f\phi_2}{(\grad^f\phi_1)^3} \Delta^f\phi_1 .
	\end{equation}
\end{proposition}
Notice that, if $f\in\Co^1(\R^2)$, then $\bar f\in\Co^1(\R^2)$ as well.
\begin{remark}\label{rem10261125}
	If $\{\phi^\epsilon\}_{\epsilon>0}$ is a smooth one-parameter family of diffeomorphisms $\phi^\epsilon:\R^2\to\R^2$ with $\phi^0=\Id$, then, for $\epsilon>0$ small enough, the functions $f_\epsilon$ defined by
	\[
	f_\epsilon\circ\phi^\epsilon = \frac{\grad^f\phi_2^\epsilon}{\grad^f\phi_1^\epsilon} .
	\]
	belong to $\Cow$ and converge to $f$ in $\Cow$.
\end{remark}

\begin{proof}
	The idea is to transform via $\phi$ the integral curves of $\grad^f$ into the ones of $\grad^{\bar f}$.
	Fix $p=(\eta ,\tau )$, let $q:=(\bar \eta ,\bar \tau ):=\phi(p)$ and let $g_p:I\to\R$ be like in Lemma \ref{lem11190949}.
	Thanks to the condition $\grad^f\phi_1\neq0$ and the Implicit Function Theorem, there exist two $\Co^2$-function $s:I\to\R$ and $\bar g_q:s(I)\to\R$, such that 
	\[
	q+(s,\bar g_q(s)) = \phi(p+(t,g_p(t)) ,
	\qquad \forall t\in I.
	\]
	Therefore
	\[
	\begin{cases}
	s(t) = \phi_1(\eta +t,\tau +g_p(t)) - \bar \eta  \\
	\bar g_q(s(t)) = \phi_2(\eta +t,\tau +g_p(t)) - \bar \tau  .
	\end{cases}
	\]
	
	We define
	\[
	\bar f(q) := \bar g_q'(0) .
	\]
	Notice that this value does not depend on the choice of $g_p$, as far as $t\mapsto (t,g_p(t))$ is an integral curve of $\grad^f$.

	We want to write $\bar g_q'(0)$.
	Set
	\[
	p_t := (\eta +t,\tau +g_p(t)).
	\]
	First
	\begin{equation*}
	\frac{\dd}{\dd t}s(t) 
	= \de_\eta \phi_1(p_t) + \de_\tau \phi_1(p_t) g_p'(t) 
	= \grad^f\phi_1(p_t) ,
	\end{equation*}
	\[
	\frac{\dd}{\dd t} \bar g_q(s(t))
	= \de_\eta \phi_2(p_t) + \de_\tau \phi_2(p_t) g_p'(t) 
	= \grad^f\phi_2(p_t) .
	\]
	Since
	\begin{equation*}
	\frac{\dd}{\dd t} \bar g_q(s(t)) = \bar g_q'(s(t))\cdot \frac{\dd}{\dd t}s(t) ,
	\end{equation*}
	we have for $s=0=t$
	\[
	\bar f(q) = \frac{ \grad^f\phi_2(p) }{ \grad^f\phi_1(p) } .
	\]
	
	$\grad^{\bar f}\bar f(q)$ is the derivative of $\bar f$ along the curve $q+(s,\bar g_q(s))$ at $s=0$, i.e.,
	\[
	\grad^{\bar f}\bar f(q) = \bar g_q''(0) .
	\]
	As above, we want to write down $\bar g_q''(0)$ in a more explicit way.
	\begin{multline*}
	\frac{\dd^2}{\dd t^2}s(t)|_{t=0} 
	= \de_\eta ^2\phi_1(p) + \de_\tau \de_\eta \phi_1(p) f(p) + \hfill \\ \hfill
		+ \de_\eta \de_\tau \phi_1(p) f(p) + \de_\tau ^2\phi_1(p) (f(p))^2 + \de_\tau \phi_1(p) \grad^ff(p) 
	= \Delta^f\phi_1(p) . 
	\end{multline*}
	\begin{multline*}
	\frac{\dd^2}{\dd t^2} \bar g_q(s(t))|_{t=0}
	= \de_\eta ^2\phi_2(p) + \de_\tau \de_\eta \phi_2(p) f(p) + \hfill \\ \hfill
		+ \de_\eta \de_\tau \phi_2(p) f(p) + \de_\tau ^2\phi_2(p) (f(p))^2 + \de_\tau \phi_2(p) \grad^ff(p) 
	= \Delta^f\phi_2(p)
	\end{multline*}
	
	Since
	\[
	\frac{\dd^2}{\dd t^2} \bar g_q(s(t)) 
	= \bar g_q''(s(t))\cdot \left(\frac{\dd}{\dd t}s(t) \right)^2 + \bar g_q'(s(t))\cdot \frac{\dd^2}{\dd t^2}s(t) ,
	\]
	we have
	\begin{align*}
	\grad^{\bar f}\bar f(q) 
	&= \bar g_q''(0)
	= \frac{ \frac{\dd^2}{\dd t^2} \bar g_q|_{t=0} - \bar g_q'(0)\cdot \frac{\dd^2}{\dd t^2}s|_{t=0} }
		{(\frac{\dd}{\dd t}s|_{t=0})^2 } \\
	&= \frac1{(\grad^f\phi_1(p))^2 } \cdot \left(
		\Delta^f\phi_2(p)
		- \frac{\grad^f\phi_2(p)}{\grad^f\phi_1(p)} \cdot \Delta^f\phi_1(p) 
		\right) .
	\end{align*}
	By Lemma \ref{lem11190949}, the function $\bar f$ belongs to $\Cow$.
\end{proof}

\section{Contact transformations}\label{sec11241822}

A diffeomorphism $\Phi:\HH\to\HH$ is a \emph{contact diffeomorphism} if $\dd \Phi(H)\subset H$, see~\cite{MR2312336,MR1317384}.
Contact diffeomorphisms are the only diffeomorphisms that preserve the sub-Riemannian perimeter.

\begin{proposition}\label{prop11202024}
	Let $\Phi:\HH\to\HH$ be a diffeomorphism of class $\Co^2$.
	If, for all $E\subset\HH$ measurable and all $\Omega\subset\HH$ open, it holds
	\begin{equation}\label{eq01221851}
	P(E;\Omega)<\infty \qquad\THEN\qquad P(\Phi(E);\Phi(\Omega))<\infty,
	\end{equation}
	then $\Phi$ is contact.
\end{proposition}

We will show in this section that any variation of an intrinsic graph $\Gamma_f$ via contact diffeomorphisms is equivalent to a variation of $f$ via the transformations of Proposition~\ref{prop09141457} and Remark~\ref{rem10261125}.

\begin{proposition}\label{prop11191058}
	Let $\phi:\R^2\to\R^2$ be a $\Co^\infty$-diffeomorphism and $f,\bar f\in{\Cow}$ as in Proposition \ref{prop09141457}.
	Then there is a contact diffeomorphism $\Phi:\Omega\to\Phi(\Omega)$, where $\Omega$ and $\Phi(\Omega)$ are open subsets of $\HH$ with $\Gamma_f\subset\Omega$, such that $\Phi(\Gamma_f)=\Gamma_{\bar f}$.
\end{proposition}

\begin{proposition}\label{prop11191059}
	Let	$\Phi^\epsilon:\HH\to\HH$, $\epsilon\in\R$, be a smooth one-parameter family of contact diffeomorphisms such that there is a compact set $K\subset\HH$ with $\Phi^\epsilon|_{\HH\setminus K}=\Id$ for all $\epsilon$ and $\Phi^0=\Id$.
	Let $f\in\Co^\infty(\R^2)$.
	Then there is $\epsilon_0>0$ such that for all $\epsilon$ with $|\epsilon|<\epsilon_0$, 
	the maps $\phi^\epsilon:\R^2\to\R^2$,
	\[
	\phi^\epsilon(p) :=\pi_{X}\circ\Phi^\epsilon (p*f(p)),
	\]
	form a smooth family of $\Co^\infty$-diffeomorphism of $\R^2$.
	
	Moreover, if $f^\epsilon$ is the function defined via $f$ and $\phi^\epsilon$ as in Proposition~\ref{prop09141457}, then 
	\[
	\Phi^\epsilon(\Gamma_f) = \Gamma_{f^\epsilon} .
	\]
\end{proposition}

\subsection{Proof of Proposition \ref{prop11202024}}
We use an argument by contradiction. 
Assume that $\Phi$ is not a contact diffeomorphism.
Then there is an open and bounded set $\Omega\subset\HH$ such that for all $p\in\Omega$ it holds $\dd\Phi(H_p) \not\subset H_{\Phi(p)}$.
Thanks to the following lemma and Remark~\ref{rem10261128}, we get a contradiction with the property~\eqref{eq01221851}.

\begin{lemma}\label{lem10261002}
	Let $\Phi:\HH\to\HH$ be a diffeomorphism of class $\Co^2$.
	Let $\Omega\subset\HH$ be an open and bounded set such that for all $p\in\Omega$
	\[
	\dd\Phi(H_p) \not\subset H_{\Phi(p)} .
	\]
	Let $E\subset\HH$ be measurable.
	If $P(E;\Omega)<\infty$ and $P(\Phi(E);\Phi(\Omega))<\infty$, then $E$ has finite Riemannian perimeter in $\Omega$.
\end{lemma}
\begin{proof}
	We extend the scalar product  $\langle \cdot,\cdot \rangle$ to the whole $T\HH$ 
	in such a way that 
	$(X,Y,Z)$ is an orthonormal frame.
	The Riemannian perimeter is defined as
	\[
	P_{\scr R}(E;\Omega):=
	\sup\left\{ \int_E\div U\dd\scr L^3:\ U\in\Vec(T\HH),\ \spt U\ssubset\Omega,\ \|U\|\le1\right\} .
	\]
	
	Let $U\in\Vec(T\HH)$ with $\spt(U)\ssubset\Omega$ and $\|U\|\le 1$.
	Then there are $V,W\in\Vec(T\HH)$ with $\spt(V)\cup\spt(W)=\spt(U)$, $V+W=U$, $V(p)\in H_p$ for all $p$, $\|V\|\le K$ and $\|W\|\le K$, and $\Phi_*W(p)\in H_p$ for all $p$, where $K\ge0$ depends on $\Phi$ and $\Omega$, but not on $U$.
	
	Remind that, if $W$ is a smooth vector field on $\HH$, then\footnote{A sketch of the proof of this formula: it is clearer to show the dual formula $\div(\Phi^*W) = \div(W)\circ\Phi \cdot J(\Phi)$; consider $W$ as a 2-form and the divergence as the exterior derivative $\dd$; remind that $\dd\Phi^*=\Phi^*\dd$; the formula follows.}
	\[
	\div(\Phi_*W) = \div(W)\circ\Phi^{-1} \cdot J(\Phi^{-1}) .
	\]
	Therefore 
	$\int_E\div W\dd\LL^3 
	= \int_{\Phi(E)}(\div W)\circ\Phi^{-1} J\Phi^{-1}\dd\LL^3
	= \int_{\Phi(E)}\div(\Phi_*W) \dd \LL^3$.
	Moreover, since $\Omega$ is bounded, 
	we can assume $\|\dd\Phi(v)\|\le K\|v\|$ for all $v\in T\Omega$, where $K\ge0$ is the same constant as above.
	Therefore
	\begin{align*}
	\int_E \div U\dd\LL^3
	&= \int_E \div V\dd\LL^3 + \int_E \div W\dd\LL^3 \\
	&= \int_E \div V\dd\LL^3 + \int_{\Phi(E)}\div(\Phi_*W) \dd \LL^3 \\
	&\le K P(E;\Omega) + K^2 P(\Phi(E);\Phi(\Omega)).
	\end{align*}
	This implies that $P_{\scr R}(E;\Omega)\le K P(E;\Omega) + K^2 P(\Phi(E);\Phi(\Omega))<\infty$.
\end{proof}

\subsection{Proof of Proposition \ref{prop11191058}}\label{sec10261906}
In this case our choice of coordinates is not helpful. 
So, we consider the exponential coordinates of second kind $(\xi,\eta,\tau)\mapsto \exp(\eta B+\tau C)*\exp(\xi A)$, using the notation of Section~\ref{subs11191130}. 

We define the map $\Phi$ as
\[
\Phi\left( \xi,\eta,\tau \right) := 
\left(  \frac{\grad^\xi\phi_2}{\grad^\xi\phi_1} (\eta,\tau),\phi_1(\eta,\tau),\phi_2(\eta,\tau) \right)
\]
Clearly, $\Phi$ is well defined and smooth on the open set 
\[
\Omega:= \{( \xi,\eta,\tau): \grad^\xi\phi_1 (\eta,\tau) \neq0 \},
\]
$\Gamma_f\subset\Omega$ by the hypothesis of Proposition \ref{prop09141457} and
$\Phi(\Gamma_f) = \Gamma_{\bar f}$.
In these coordinates, the differential of $\Phi$ is
\[
\dd \Phi(\xi,\eta,\tau)
= 
\begin{pmatrix}
 	\de_\xi\left( \frac{\grad^\xi\phi_2}{\grad^\xi\phi_1} \right)
		& \de_\eta\left( \frac{\grad^\xi\phi_2}{\grad^\xi\phi_1} \right)
		& \de_\tau\left( \frac{\grad^\xi\phi_2}{\grad^\xi\phi_1} \right) \\
	0 & \de_\eta\phi_1 & \de_\tau\phi_1 \\
	0 & \de_\eta\phi_2 & \de_\tau\phi_2 
\end{pmatrix}
\]
Since $\phi$ is a diffeomorphism, $\Phi$ is a diffeomorphism if and only if $\de_\xi\left( \frac{\grad^\xi\phi_2}{\grad^\xi\phi_1} \right)\neq0$.
A short computation shows that
\[
\de_\xi\left( \frac{\grad^\xi\phi_2}{\grad^\xi\phi_1} \right)
= \frac{ \det(\dd\phi) }{ (\grad^\xi\phi_1)^2 },
\]
which is  non-zero.

Now, we need to show that $\Phi$ is a contact diffeomorphism.
In this system of coordinates, the left-invariant vector fields $X,Y,Z$ are written as
\[
\tilde X(\xi,\eta,\tau) = \de_\xi,
\qquad
\tilde Y(\xi,\eta,\tau) = \de_\eta + \xi\de_\tau,
\qquad
\tilde Z(\xi,\eta,\tau) = \de_\tau .
\]
We have
\[
\dd\Phi \left(\tilde X(\xi,\eta,\tau)\right)
= \de_\xi\left( \frac{\grad^\xi\phi_2}{\grad^\xi\phi_1} \right) \tilde X(\Phi(\xi,\eta,\tau)) ,
\]
\[
\dd \Phi \left(\tilde Y(\xi,\eta,\tau)\right)
= \grad^\xi \left( \frac{\grad^\xi\phi_2}{\grad^\xi\phi_1} \right) \tilde X(\Phi(\xi,\eta,\tau))
	+ \grad^\xi\phi_1 \tilde Y(\Phi(\xi,\eta,\tau)) .
\]
Therefore, $\dd\Phi(H)\subset H$. 
\qed

\subsection{Proof of Proposition \ref{prop11191059}}

The functions $\phi^\epsilon:\R^2\to\R^2$ are well defined and smooth for all $\epsilon\in\R$.
Since $\Phi^\epsilon$ and all its derivative converge to $\Id$ uniformly on $\HH$, there exists $\epsilon_0>0$ such that for all $\epsilon$ with $|\epsilon|<\epsilon_0$, the vector field $X$ is not tangent to $\Phi^\epsilon(\Gamma_f)$ at any point.
Therefore, $\det(\dd\phi^\epsilon)\neq0$ for all such $\epsilon$.
Since $\phi^\epsilon|_{\pi_X(K)} =\Id$, $\phi^\epsilon$ is a covering map and therefore it is a smooth diffeomorphism.

The last statement is a direct consequence Lemma \ref{lem11191340}.
\qed

\section{First Contact Variation}\label{sec10262029}
Similar formulas for the first and the second variation for the sub-Riemannian perimeter  in the Heisenberg group can be found in \cite{MR2472175,MR3558526,MR3044134,Montefalcone2015Intrinsic-varia}.

In all the formulas below, we set $\psi:=\grad^ff$.
\begin{proposition}\label{prop11170018}
	Let $f\in\Cow$ be such that $\Gamma_f$ is an area-minimizing surface.
	Then for all $V_1,V_2\in\Co^\infty_c(\R^2)$ it holds
	\begin{equation}\label{eq11162345}
	0= \int_{\R^2} \bigg[ \frac{\psi}{\sqrt{1+\psi^2}}\left( 
		-2 \psi \cdot \grad^fV_1 
		- f \cdot \Delta^fV_1 \right)
		+ \sqrt{1+\psi^2} \de_\eta V_1 
	\bigg]\dd\eta\dd\tau .
	\end{equation}
	and
	\begin{equation}\label{eq11162346}
	0 = \int_{\R^2}\bigg[ \frac{\psi}{\sqrt{1+\psi^2}} \Delta^fV_2
		+ \sqrt{1+\psi^2} \de_\tau V_2
		\bigg]\dd\eta\dd\tau .
	\end{equation}
\end{proposition}
\begin{proposition}\label{prop11170019}
	Let $f\in\Co^\infty(\R^2)$ be such that for all $V_2\in\Co^\infty_c(\R^2)$ the equation \eqref{eq11162346} holds.
	Then \eqref{eq11162345} holds as well for all $V_1\in\Co^\infty_c(\R^2)$.
\end{proposition}
\begin{proposition}\label{prop11170020}
	A function $f\in\Co^\infty(\R^2)$ satisfies \eqref{eq11162346}  for all $V_2\in\Co^\infty_c(\R^2)$ if and only if
	\begin{equation}\label{eq10151603}
		(\grad^f + 2\de_\tau f)
		\grad^f
		\left(\frac{\psi}{\sqrt{1+\psi^2}}\right) = 0 .
	\end{equation}
\end{proposition}

\subsection{Proof of Proposition \ref{prop11170018}}\label{sec11171157}
Let $f\in\Cow$, $\omega\subset\R^2$ an open and bounded set and
 $V=(V_1,V_2):\R^2\to\R^2$ a smooth vector field with $\spt V\ssubset\omega$.
Let $\phi^\epsilon:\R^2\to\R^2$ be a smooth one-parameter family of diffeomorphism such that $\{\phi^\epsilon\neq\Id\}\subset\spt V$ for all $\epsilon>0$ and, for all $p\in\R^2$,
\[
\begin{cases}
 	\phi^0(p)=p \\
	\de_\epsilon\phi^\epsilon(p)|_{\epsilon=0} = V(p) .
\end{cases}
\]	
	
Notice that $\nabla^f\phi^\epsilon_1 = \de_\eta \phi^\epsilon_1 + f \de_\tau\phi^\epsilon_1$ is not zero for $\epsilon$ small enough, because $\nabla^f\phi^\epsilon_1$ converges to 1 uniformly as $\epsilon\to0$.
Hence, by Proposition \ref{prop09141457}, there is an interval $I=(-\hat\epsilon,\hat\epsilon)$ such that the function given by
\begin{equation}\label{10241847}
	f_\epsilon\circ\phi^\epsilon = \frac{\grad^f\phi_2^\epsilon}{\grad^f\phi_1^\epsilon}
\end{equation}
is well defined for all $\epsilon\in I$.
Define $\gamma:I\to\R$ as
\begin{align*}
\gamma(\epsilon) &:= \int_\omega \sqrt{1+(\nabla^{f_\epsilon}f_\epsilon)^2} \dd \eta \dd \tau  \\
&= \int_\omega \sqrt{1+((\nabla^{f_\epsilon}f_\epsilon)\circ\phi^\epsilon)^2} J_{\phi^\epsilon} \dd \eta \dd \tau,
\end{align*}
where we performed a change of coordinates via $\phi^\epsilon$ and 
\[
J_{\phi^\epsilon} = \de_\eta \phi^\epsilon_1 \de_\tau\phi^\epsilon_2 - \de_\tau\phi^\epsilon_1 \de_\eta \phi^\epsilon_2 
\]
is the Jacobian of $\phi^\epsilon$.
Using equality~\eqref{eq10241705} and Lemma~\ref{lem10241748}, it is immediate to see that $\gamma$ is continuous.

\begin{lemma}\label{lem10251040}
	The function $\gamma:I\to\R$ is continuously differentiable and
\begin{multline}\label{eq10241902}
 	\gamma'(\epsilon) =
	 \int_\omega \bigg[ \frac{((\nabla^{f_{\epsilon}}f_{\epsilon})\circ\phi^\epsilon)}{\sqrt{1+((\nabla^{f_{\epsilon}}f_{\epsilon})\circ\phi^\epsilon)^2}} A_f(\epsilon)  J_{\phi^\epsilon} 
	+ \\
	+ \sqrt{1+((\nabla^{f_{\epsilon}}f_{\epsilon})\circ\phi^\epsilon)^2} \de_\epsilon J_{\phi^\epsilon}
	\bigg] \dd \eta \dd \tau ,
\end{multline}
where
\begin{multline}\label{eq10251042}
A_f(\epsilon) 
:= \frac{\Delta^f\de_\epsilon\phi^\epsilon_2}{(\grad^f\phi^\epsilon_1)^2}
	-2 \frac{\Delta^f\phi^\epsilon_2}{(\grad^f\phi^\epsilon_1)^3} \grad^f\de_\epsilon\phi^\epsilon_1 +\\ 
	- \frac{\grad^f\de_\epsilon\phi^\epsilon_2}{(\grad^f\phi^\epsilon_1)^3} \Delta^f\phi^\epsilon_1
 	+3 \frac{\grad^f\phi^\epsilon_2}{(\grad^f\phi^\epsilon_1)^4} \grad^f\de_\epsilon\phi^\epsilon_1\cdot \Delta^f\phi^\epsilon_1
	- \frac{\grad^f\phi^\epsilon_2}{(\grad^f\phi^\epsilon_1)^3} \Delta^f\de_\epsilon\phi^\epsilon_1 .
\end{multline}
\end{lemma}
\begin{proof}[Proof of Lemma~\ref{lem10251040}]
	First, suppose $f\in\Co^\infty(\R^2)$.
	Then $\gamma\in\Co^\infty(I)$ and
	\begin{multline*}
	\gamma'(\epsilon) 
	= \int_\omega\bigg[ \frac{((\nabla^{f_\epsilon}f_\epsilon)\circ\phi^\epsilon)}{\sqrt{1+((\nabla^{f_\epsilon}f_\epsilon)\circ\phi^\epsilon)^2}}\de_\epsilon((\nabla^{f_\epsilon}f_\epsilon)\circ\phi^\epsilon) J_{\phi^\epsilon} 
		+ \\
		+ \sqrt{1+((\nabla^{f_\epsilon}f_\epsilon)\circ\phi^\epsilon)^2} \de_\epsilon J_{\phi^\epsilon}
		\bigg]\dd \eta \dd \tau .
	\end{multline*}
	Applying the formula in Proposition \ref{prop09141457} and the identity $\grad^f\de_\epsilon = \de_\epsilon\grad^f$, one obtains
	\[
	\de_\epsilon((\nabla^{f_\epsilon}f_\epsilon)\circ\phi^\epsilon) 
	= A_f(\epsilon)
	\]
	and thus formula~\eqref{eq10241902} holds in the smooth case.
	
	Next, suppose $f=f_\infty$ is the limit in $\Cow$ of a sequence $f_k\in\Co^\infty(\R^2)$, as in Lemma~\ref{lem10241912}.
	Notice that $\nabla^{f_k}\phi^\epsilon_1$ is not zero for $\epsilon$ small enough and $k$ large enough.
	Indeed, $|\grad^{f_k}\phi^\epsilon_1-\grad^{f_\infty}\phi^\epsilon_1| \le \|f_k-f\|_{\LL^\infty(\spt V)} \|\de_\tau\phi_1^\epsilon\|_{\LL^\infty(\spt V)}$ and $\grad^{f_\infty}\phi^\epsilon_1$ converges to one uniformly on $\R^2$ as $\epsilon\to0$.
	Hence, there is an interval $I\subset\R$ centered at zero such that
	the functions $f_{k,\epsilon}$ as in Proposition~\ref{prop09141457} are well defined for $\epsilon\in I$ and, without loss of generality, for all $k\in\N\cup\{\infty\}$.
	For $k\in\N\cup\{\infty\}$, define $\gamma_k:I\to\R$ as
	\[
	\gamma_k(\epsilon) := \int_\omega \sqrt{1+(\grad^{f_{k,\epsilon}}f_{k,\epsilon})^2} \dd \eta \dd \tau  
	\]
	Define also the function $\eta:I\to\R$ as the right-hand side of \eqref{eq10241902}.
	From Lemma~\ref{lem10241748}, it follows that $\{A_{f_k}\}_{k\in\N}$ converges to $A_f$ uniformly on $I$.
	Therefore, we have that $\{\gamma_k\}_{k\in\N}$ and $\{\gamma_k'\}_{k\in\N}$ converge to $\gamma$ and $\eta$ uniformly on $I$.
	We conclude that $\gamma\in\Co^1(I)$ and $\gamma'=\eta$.
\end{proof}

In order to evaluate $\gamma'(0)$, notice that 
\begin{align*}
\grad^f\phi^0_1 &= 1 
	& \grad^f\phi^0_2 &= f \\
\grad^f\de_\epsilon\phi^\epsilon_1|_{\epsilon=0} &= \grad^fV_1 
	& \grad^f\de_\epsilon\phi^\epsilon_2|_{\epsilon=0} &= \grad^fV_2 \\
\Delta^f\phi^0_1 &= 0 
	& \Delta^f\phi^0_2 &= \psi \\
\Delta^f\de_\epsilon\phi^\epsilon_1 |_{\epsilon=0} &= \Delta^f V_1 
	& \Delta^f\de_\epsilon\phi^\epsilon_2 |_{\epsilon=0} &= \Delta^f V_2 .
\end{align*}
Therefore
\[
A_f(0) 
= \Delta^fV_2 
-2 \psi \grad^fV_1 
- f \Delta^fV_1 .
\]
Moreover, using the facts $\de_\tau\phi^0_1 = \de_\eta \phi^0_2 = 0$ and $\de_\eta \phi^0_1 = \de_\tau\phi^0_2 = 1$ and that the derivatives $\de_\epsilon$, $\de_\eta $ and $\de_\tau$ commute, we have
\[
\de_\epsilon J_{\phi^\epsilon}|_{\epsilon=0} 
= \de_\eta V_1 + \de_\tau V_2 .
\]
Putting all together, we obtain
\begin{multline*}
\gamma'(0) 
= \int_\omega \bigg[ \frac{\psi}{\sqrt{1+\psi^2}}\left(\Delta^fV_2
	-2 \psi \cdot \grad^fV_1 
	- f \cdot \Delta^fV_1 \right)
	+  \\ 
	+ \sqrt{1+\psi^2} (\de_\eta V_1 + \de_\tau V_2)
	\bigg] \dd \eta \dd \tau .
\end{multline*}
Since $\Gamma_f$ is an area-minimizing surface, then $\gamma'(0)=0$ for all $V_1,V_2\in\Co^\infty_c(\R^2)$.
Since this expression is linear in $V$, then we obtain both conditions \eqref{eq11162345} and \eqref{eq11162346}.
\qed

\subsection{Proof of Proposition \ref{prop11170019}}
Let $V_1\in\Co^\infty(\R^2)$ and set $V_2:=fV_1\in\Co^\infty_c(\R^2)$.
Then
\begin{multline*}
0 = \int_{\R^2}\bigg[ \frac{\psi}{\sqrt{1+\psi^2}} \Delta^fV_2 
		+ \sqrt{1+\psi^2} \de_\tau V_2
		\bigg]\dd\eta\dd\tau \\
= \int_{\R^2}\bigg[ \frac{\psi}{\sqrt{1+\psi^2}} (\grad^f\psi V_1 + 2\psi\grad^fV_1 + f\Delta^fV_1) + \hfill \\ \hfill
		+ \sqrt{1+\psi^2} (\de_\tau fV_1 + f \de_\tau V_1)
		\bigg]\dd\eta\dd\tau \\
= \int_{\R^2}\bigg[ \frac{\psi}{\sqrt{1+\psi^2}}   (2\psi\grad^fV_1 + f\Delta^fV_1) + \hfill \\ 
		+ \left(\frac{\psi\grad^f\psi}{\sqrt{1+\psi^2}} +\sqrt{1+\psi^2} \de_\tau f \right) V_1 + \\ \hfill
		+ \sqrt{1+\psi^2} (\grad^fV_1 - \de_\eta V_1)
		\bigg]\dd\eta\dd\tau \\
= \int_{\R^2}\bigg[ \frac{\psi}{\sqrt{1+\psi^2}}   (2\psi\grad^fV_1 + f\Delta^fV_1)
		- \sqrt{1+\psi^2} \grad^fV_1 + \hfill \\ \hfill
		+ \sqrt{1+\psi^2} (\grad^fV_1 - \de_\eta V_1)
		\bigg]\dd\eta\dd\tau \\
= \int_{\R^2}\bigg[ \frac{\psi}{\sqrt{1+\psi^2}}   (2\psi\grad^fV_1 + f\Delta^fV_1)
		- \sqrt{1+\psi^2}  \de_\eta V_1)
		\bigg]\dd\eta\dd\tau .\hfill
\end{multline*}
Hence \eqref{eq11162345} holds true for $V_1$ as well.
\qed

\subsection{Proof of Proposition \ref{prop11170020}}
We have for all $V_2\in\Co^\infty_c(\R^2)$
\begin{multline*}
\int_{\R^2}\left[ \frac{\psi}{\sqrt{1+\psi^2}} \grad^f\grad^fV_2 + \sqrt{1+\psi^2} \de_\tau V_2 \right]\dd\eta\dd\tau \\
= -\int_{\R^2}\left[ 
	\grad^f\left(\frac{\psi}{\sqrt{1+\psi^2}}\right) \grad^fV_2 
	+ \frac{\de_\tau f\psi}{\sqrt{1+\psi^2}} \grad^fV_2
	+ \de_\tau (\sqrt{1+\psi^2}) V_2 \right]\dd\eta\dd\tau \\
= \int_{\R^2}\left[ 
	\grad^f\grad^f\left(\frac{\psi}{\sqrt{1+\psi^2}}\right) V_2 
	+ \de_\tau f\grad^f\left(\frac{\psi}{\sqrt{1+\psi^2}}\right) V_2 + \right.\hfill\\
	+ \grad^f(\de_\tau f) \frac{\psi}{\sqrt{1+\psi^2}} V_2
	+ \de_\tau f \grad^f\left(\frac{\psi}{\sqrt{1+\psi^2}}\right) V_2 
	+ (\de_\tau f)^2 \frac{\psi}{\sqrt{1+\psi^2}}V_2 +\\ \hfill\left.
	- \frac{\psi}{\sqrt{1+\psi^2}} \de_\tau \psi V_2 \right]\dd\eta\dd\tau .
\end{multline*}
Therefore, using the fact that $\de_\tau \psi = \grad^f(\de_\tau f) + (\de_\tau f)^2$, we get that \eqref{eq11162346} is equivalent to
\begin{equation*}
 	 \grad^f\grad^f\left(\frac{\psi}{\sqrt{1+\psi^2}}\right)  
	+ 2\de_\tau f\cdot\grad^f\left(\frac{\psi}{\sqrt{1+\psi^2}}\right) = 0 .
\end{equation*}
\qed

\section{Second Contact Variation}\label{sec10262031}

Similarly to the previous sections, we set $\psi:=\grad^ff$.

\begin{proposition}\label{prop11171154}
	If the intrinsic graph of $f\in\Cow$ is an area-minimizing surface, then, for all $V_1,V_2\in\Co^\infty_c(\R^2)$, we have:
	\begin{multline}\label{eq11170006}
	0\le \II_f(V_1,V_2) :=
	\int_{\R^2}\bigg[  \frac{  (\Delta^fV_2 -2 \psi \grad^fV_1 - f \Delta^fV_1)^2 }{ (1+\psi^2)^{\frac32} } +  \\
	+ \frac{ \psi }{ (1+\psi^2)^{\frac12} } 
	\left(
	- 4  \Delta^f V_2 \cdot \grad^f V_1  
	- 2 \grad^f V_2 \cdot \Delta^f V_1 +  \right.\hfill\\ \hfill\left.
	+ 6  f \cdot \grad^f V_1 \cdot \Delta^f V_1  
	+ 6  \psi \cdot (\grad^f V_1)^2  
	\right) + \\
	+ 2 \frac{ \psi }{ (1+\psi^2)^{\frac12} }
		 (\Delta^fV_2 -2 \psi \grad^fV_1 - f \Delta^fV_1)  (\de_\eta V_1 + \de_\tau V_2) +\\
	+ 2 (1+\psi^2)^{\frac12}   (\de_\eta V_1 \de_\tau V_2 - \de_\tau V_1 \de_\eta V_2) 
	\bigg]\dd\eta\dd\tau .
	\end{multline}
\end{proposition}

\subsection{Proof of Proposition \ref{prop11171154}}
Let $\omega\subset\R^2$ be an open and bounded set and
 $V=(V_1,V_2):\R^2\to\R^2$ a smooth vector field with $\spt V\ssubset\omega$.
Let $\phi^\epsilon=(\phi^\epsilon_1,\phi^\epsilon_2):\R^2\to\R^2$ be a smooth one-parameter family of diffeomorphism such that 
$\{\phi^\epsilon\neq\Id\}\subset\spt V$ for all $\epsilon>0$ and, for all $p\in\R^2$,
\[
\begin{cases}
 	\phi^0(p)=p \\
	\de_\epsilon\phi^\epsilon(p)|_{\epsilon=0} = V(p) .
\end{cases}
\]	
Define $W_i(p) := \de_\epsilon^2\phi^\epsilon_i(p)|_{\epsilon=0}$.
Then $W=(W_1,W_2):\R^2\to\R^2$ is a smooth vector field with $\spt W\ssubset\omega$.

As for the first variation, see Section \ref{sec11171157}, define
\[
\gamma(\epsilon) := \int_\omega \sqrt{1+(\nabla^{f_\epsilon}f_\epsilon)^2} \dd \eta \dd \tau . 
\]

\begin{lemma}\label{lem08202239}
	The function $\gamma:I\to\R$ is twice continuously differentiable and 
	\begin{multline}\label{eq10251033}
	\gamma''(\epsilon) = \int_\omega \bigg[  \frac{  A_f(\epsilon)^2 }{ (1+(\grad^{f_\epsilon}f_\epsilon\circ\phi^\epsilon)^2)^{\frac32} } J_{\phi^\epsilon} 
	+ \frac{ (\grad^{f_\epsilon}f_\epsilon\circ\phi^\epsilon) B_f(\epsilon) }{ (1+(\grad^{f_\epsilon}f_\epsilon\circ\phi^\epsilon)^2)^{\frac12} } J_{\phi^\epsilon} + \\
	+ 2 \frac{ (\grad^{f_\epsilon}f_\epsilon\circ\phi^\epsilon) A_f(\epsilon) }{ (1+(\grad^{f_\epsilon}f_\epsilon\circ\phi^\epsilon)^2)^{\frac12} } \de_\epsilon J_{\phi^\epsilon} 
	+ (1+(\grad^{f_\epsilon}f_\epsilon\circ\phi^\epsilon)^2)^{\frac12} \de_\epsilon^2 J_{\phi^\epsilon} \bigg]\dd y\dd z ,
	\end{multline}
	where $A_f(\epsilon)$ is defined as in \eqref{eq10251042} and
	\begin{multline*}
	B_f(\epsilon) := 
	\frac{ \Delta^f\de_\epsilon^2\phi_2^\epsilon }{ (\grad^f\phi^\epsilon_1)^2 } 
	- 2 \frac{ \Delta^f\de_\epsilon \phi^\epsilon_2 \cdot \grad^f\de_\epsilon\phi_1^\epsilon }{ (\grad^f\phi^\epsilon_1)^3 } + \\
	- 2 \frac{ \Delta^f\de_\epsilon \phi^\epsilon_2 \cdot \grad^f\de_\epsilon\phi_1^\epsilon }{ (\grad^f\phi^\epsilon_1)^3 }  
	- 2  \frac{ \Delta^f \phi^\epsilon_2 \cdot \grad^f\de_\epsilon^2\phi_1^\epsilon }{ (\grad^f\phi^\epsilon_1)^3 } 
	+ 6 \frac{ \Delta^f \phi^\epsilon_2 \cdot (\grad^f\de_\epsilon\phi_1^\epsilon)^2 }{ (\grad^f\phi^\epsilon_1)^4 } + \\
	-  \frac{ \grad^f\de_\epsilon^2 \phi^\epsilon_2 \cdot \Delta^f\phi_1^\epsilon }{ (\grad^f\phi^\epsilon_1)^3 } 
	- \frac{ \grad^f\de_\epsilon \phi^\epsilon_2 \cdot \Delta^f\de_\epsilon\phi_1^\epsilon }{ (\grad^f\phi^\epsilon_1)^3 } 
	+ 3 \frac{ \grad^f\de_\epsilon \phi^\epsilon_2 \cdot \Delta^f\phi_1^\epsilon \cdot \grad^f\de_\epsilon\phi_1^\epsilon}{ (\grad^f\phi^\epsilon_1)^4 } + \\
	+ 3 \frac{ \grad^f\de_\epsilon \phi^\epsilon_2 \cdot \grad^f\de_\epsilon\phi_1^\epsilon \cdot \Delta^f\phi_1^\epsilon }{ (\grad^f\phi^\epsilon_1)^4 } 
	+ 3 \frac{ \grad^f \phi^\epsilon_2 \cdot \grad^f\de_\epsilon^2\phi_1^\epsilon \cdot \Delta^f\phi_1^\epsilon }{ (\grad^f\phi^\epsilon_1)^4 } +\hfill \\
	\hfill + 3 \frac{ \grad^f \phi^\epsilon_2 \cdot \grad^f\de_\epsilon\phi_1^\epsilon \cdot \Delta^f\de_\epsilon\phi_1^\epsilon }{ (\grad^f\phi^\epsilon_1)^4 } 
	- 12 \frac{ \grad^f \phi^\epsilon_2 \cdot (\grad^f\de_\epsilon\phi_1^\epsilon)^2 \cdot \Delta^f\phi_1^\epsilon }{ (\grad^f\phi^\epsilon_1)^5 } + \\
	- \frac{ \grad^f \de_\epsilon\phi^\epsilon_2 \cdot \Delta^f\de_\epsilon\phi_1^\epsilon }{ (\grad^f\phi^\epsilon_1)^3 }
	- \frac{ \grad^f \phi^\epsilon_2  \cdot \Delta^f\de_\epsilon^2\phi_1^\epsilon }{ (\grad^f\phi^\epsilon_1)^3 }
	+ 3 \frac{ \grad^f \phi^\epsilon_2  \cdot \Delta^f\de_\epsilon\phi_1^\epsilon \cdot \grad^f\de_\epsilon\phi_1^\epsilon}{ (\grad^f\phi^\epsilon_1)^4 } .
	\end{multline*}
\end{lemma}
\begin{proof}[Proof of Lemma~\ref{lem08202239}]
	This lemma is a continuation of Lemma~\ref{lem10251040}.
	
	First, suppose $f\in\Co^\infty(\R^2)$.
	Then, the function $\gamma$ is smooth and its second derivative is
	\begin{multline*}
	\gamma''(\epsilon)
	= \int_\omega\bigg[  \frac{  (\de_\epsilon(\grad^{f_\epsilon}f_\epsilon\circ\phi^\epsilon))^2 }{ (1+(\grad^{f_\epsilon}f_\epsilon\circ\phi^\epsilon)^2)^{\frac32} } J_{\phi^\epsilon} + \hfill \\
	+ \frac{ (\grad^{f_\epsilon}f_\epsilon\circ\phi^\epsilon) \de_\epsilon^2(\grad^{f_\epsilon}f_\epsilon\circ\phi^\epsilon) }{ (1+(\grad^{f_\epsilon}f_\epsilon\circ\phi^\epsilon)^2)^{\frac12} } J_{\phi^\epsilon} + \\
	+ 2 \frac{ (\grad^{f_\epsilon}f_\epsilon\circ\phi^\epsilon) \de_\epsilon(\grad^{f_\epsilon}f_\epsilon\circ\phi^\epsilon) }{ (1+(\grad^{f_\epsilon}f_\epsilon\circ\phi^\epsilon)^2)^{\frac12} } \de_\epsilon J_{\phi^\epsilon} +\\
\hfill	+ (1+(\grad^{f_\epsilon}f_\epsilon\circ\phi^\epsilon)^2)^{\frac12} \de_\epsilon^2 J_{\phi^\epsilon} \bigg]\dd y\dd z .
	\end{multline*}
	One checks by direct computation that
	\begin{align*}
	\de_\epsilon(\grad^{f_\epsilon}f_\epsilon\circ\phi^\epsilon) 
		&= A_f(\epsilon) ,\\
	\de_\epsilon^2(\grad^{f_\epsilon}f_\epsilon\circ\phi^\epsilon) 
		&= B_f(\epsilon) ,
	\end{align*}
	thus \eqref{eq10251033} is proven in the smooth case.
	
	Next, suppose $f=f_\infty$ is the limit in $\Cow$ of a sequence $f_k\in\Co^\infty(\R^2)$, as in Lemma~\ref{lem10241912}.
	Define $f_{k,\epsilon}$ and $I\subset\R$ and $\gamma_k:I\to\R$ as in the proof of Lemma~\ref{lem10251040}.
	Define also $\eta:I\to\R$ as the right-hand side of \eqref{eq10251033}.
	By Lemma~\ref{lem10241748}, $\{A_{f_k}\}_{k\in\N}$ and $\{B_{f_k}\}_{k\in\N}$ converge to $A_f$ and $B_f$ uniformly on $I$.
	Therefore, we have that the convergences $\gamma_k\to \gamma$ and $\gamma_k'\to\gamma'$ and $\gamma_k''\to\eta$ are uniform on $I$.
	We conclude that $\gamma\in\Co^2(I)$ and $\gamma''=\eta$.
\end{proof}

Next, one can directly check that
\begin{multline*}
\gamma''(0)
 = \int_\omega \bigg[ \frac{  (\Delta^fV_2 -2 \psi \grad^fV_1 - f \Delta^fV_1)^2 }{ (1+\psi^2)^{\frac32} } +  \\
+ \frac{ \psi }{ (1+\psi^2)^{\frac12} } 
( \Delta^fW_2 
-  f  \cdot \Delta^fW_1
- 2   \psi \cdot \grad^fW_1
- 4  \Delta^f V_2 \cdot \grad^f V_1  
- 2 \grad^f V_2 \cdot \Delta^f V_1 + \hfill \\ \hfill
+ 6  f \cdot \grad^f V_1 \cdot \Delta^f V_1  
+ 6  \psi \cdot (\grad^f V_1)^2  
) + \\
+ 2 \frac{ \psi }{ (1+\psi^2)^{\frac12} } (\Delta^fV_2 -2 \psi \grad^fV_1 - f \Delta^fV_1) (\de_\eta V_1 + \de_\tau V_2) +\\
\hfill	+ (1+\psi^2)^{\frac12}  (\de_\eta W_1 + \de_\tau W_2 + 2 (\de_\eta V_1 \de_\tau V_2 - \de_\tau V_1 \de_\eta V_2)) \bigg]\dd \eta \dd \tau  .
\end{multline*}

Finally, if $\Gamma_f$ is an area-minimizing surface, then $\gamma'(0)=0$ and $\gamma''(0)\ge 0$.
Notice that the terms containing $W_1$ and $W_2$ in the expression of $\gamma''(0)$ are zero because $\gamma'(0)=0$.
So, the second variation formula \eqref{eq11170006} is proven.
\qed

\section{Contact variations in the case $\Delta^ff=0$}\label{sec10262032}
In this final section we prove our main result.
We show that there is a quite large class of functions in $\Cow$ that satisfy both conditions on the first and second contact variation.
Since we know that the only intrinsic graphs of smooth functions that are area minimizers are the vertical planes, our result shows that variations along contact diffeomorphisms are not selective enough.

As usual, we set $\psi:=\grad^ff$.

\begin{lemma}\label{lem11171413}
	Let $f\in\Co^\infty(\R^2)$ be such that $\Delta^ff=0$.
	Then
	\begin{multline*}
	\II_f(V_1,V_2)
	= \int_{\R^2} \bigg[ \frac{  (\Delta^fV_2 -2 \psi \grad^fV_1 - f \Delta^fV_1)^2 }{ (1+\psi^2)^{\frac32} } + \\
	+ \de_\tau \left(\frac{ \psi }{ (1+\psi^2)^{\frac12} }\right)
	\left(
		\grad^fV_2 - \grad^f(fV_1)
	\right)^2 \bigg]\dd\eta\dd\tau .
	\end{multline*}
\end{lemma}
The proof is very technical and it is postponed to the last section below.

\begin{theorem}\label{teo11241753}
	Let $f\in\Cow$ be such that $\Delta^ff=0$ in weak Lagrangian sense.
	Then both equalities \eqref{eq11162345} and \eqref{eq11162346} and also the inequality \eqref{eq11170006} are satisfied for all $V_1,V_2\in\Co^\infty_c(\R^2)$.
\end{theorem}
\begin{proof}
	We first prove that both equalities \eqref{eq11162345} and \eqref{eq11162346} are satisfied.
	Let $\{f_k\}_{k\in\N}\subset\Co^\infty(\R^2)$ be a sequence converging to $f$ in $\Cow$ and such that $\Delta^{f_k}f_k=0$, as in Lemma~\ref{lem10251843}.
	Fix $V_1,V_2\in\Co^\infty_c(\R^2)$.
	Then \eqref{eq10151603} and \eqref{eq11162345} are satisfied by all $f_k$ thanks to Propositions~\ref{prop11170019}~and~\ref{prop11170020}.
	Passing to the limit $k\to\infty$, we prove that $f$ satisfies them too. 
	
	Now, we prove that the inequality \eqref{eq11170006} holds true.
	If $f\in\Co^\infty(\R^2)$, then we can apply Lemma~\ref{lem11171413}, where $\de_\tau \left(\frac{ \psi }{ (1+\psi^2)^{\frac12} }\right)= \frac{ \de_\tau \psi }{ (1+\psi^2)^{\frac32} }\ge0$
	because of Lemma~\ref{lem11191504}.
	So, \eqref{eq11170006} is proven for $f$ smooth.
	For $f\in\Cow$, let $\{f_k\}_{k\in\N}\subset\Co^\infty(\R^2)$ as in Lemma~\ref{lem10251843}.
	From Lemma~\ref{lem10241748} follows that, for fixed $V_1,V_2\in\Co^\infty_c(\R^2)$, it holds
	\[
	\lim_{k\to\infty}\II_{f_k}(V_1,V_2) = \II_f(V_1,V_2) ,
	\]
	thus $\II_f(V_1,V_2)\ge0$.
\end{proof}

\subsection{Proof of Lemma \ref{lem11171413}}

The proof of this lemma is just a computation, but quite elaborate.
For making the formulas more readable, we decided to drop the sign of integral along the proof.
In other words, all equalities in this section are meant as equalities of integrals on $\R^2$.
We will constantly use the formulas listed in Appendix~\ref{subs11241610} together with $\grad^f\psi=0$.

Before of all, we reorganise the integral in \eqref{eq11170006}:
\begin{align}  
\label{A}\tag{\textcircled{a}}	
	&\frac{  (\Delta^fV_2 -2 \psi \grad^fV_1 - f \Delta^fV_1)^2 }{ (1+\psi^2)^{\frac32} } +  \\
\label{B}\tag{\textcircled{b}}	
	&+ \frac{ \psi }{ (1+\psi^2)^{\frac12} } 
			\left(
			+ 6  f \cdot \grad^f V_1 \cdot \Delta^f V_1  
			+ 6  \psi \cdot (\grad^f V_1)^2 
			\right) \\
\label{C}\tag{\textcircled{c}}	
	&+ 2 \frac{ \psi }{ (1+\psi^2)^{\frac12} } 
			\left( -2 \psi \grad^fV_1 - f \Delta^fV_1 \right) \de_\eta V_1 \\
\label{D}\tag{\textcircled{d}}	
	&+ 2 \frac{ \psi }{ (1+\psi^2)^{\frac12} }
			\Delta^fV_2 \de_\tau V_2 \\
\label{E}\tag{\textcircled{e}}	
	&+  \frac{ \psi }{ (1+\psi^2)^{\frac12} }
			\left(
			- 4  \Delta^f V_2 \cdot \grad^f V_1  
			- 2 \grad^f V_2 \cdot \Delta^f V_1
			\right) \\
\label{F}\tag{\textcircled{f}}	
	&+	2 \frac{ \psi }{ (1+\psi^2)^{\frac12} }
			\left(
			\Delta^fV_2 \de_\eta  V_1
			+
			(-2 \psi \grad^fV_1 - f \Delta^fV_1) \de_\tau V_2
			\right)\\
\label{G}\tag{\textcircled{g}}	
	&+ 2 (1+\psi^2)^{\frac12}   (\de_\eta V_1 \de_\tau V_2 - \de_\tau V_1 \de_\eta V_2) .
\end{align}
In the following lemmas we will study \ref{B}+\ref{C}, \ref{D} and $\ref{E}+\ref{F}+\ref{G}$ separately in order to obtain the expansion of the square in the second term of the integral in Lemma \ref{lem11171413}.
\begin{lemma}\label{lem11171449}
	\begin{equation*}
	\ref{B}+\ref{C} = 
	\de_\tau \left(\frac{ \psi }{ (1+\psi^2)^{\frac12} }\right)
	\left( \grad^f(fV_1) \right)^2 .
	\end{equation*}
\end{lemma}
\begin{proof}[Proof of Lemma~\ref{lem11171449}]
	\begin{align*}
	\ref{B} &=
	\frac{ \psi }{ (1+\psi^2)^{\frac12} } 
		\left(
		 6  f  \grad^f V_1  \Delta^f V_1  
		+ 6  \psi  (\grad^f V_1)^2 
		\right) \\
	&= \frac{ \psi }{ (1+\psi^2)^{\frac12} } 
		\left(
		 3  f  \grad^f(\grad^fV_1)^2
		+ 6  \psi  (\grad^f V_1)^2 
		\right) \\
	&= 3 \frac{ \psi }{ (1+\psi^2)^{\frac12} } 
		\left(
		  \grad^f( f  (\grad^fV_1)^2)
		+   \psi  (\grad^f V_1)^2 
		\right) \\
	&= -3 \frac{ \psi }{ (1+\psi^2)^{\frac12} } f\de_\tau f (\grad^fV_1)^2
		+ 3 \frac{ \psi^2 }{ (1+\psi^2)^{\frac12} } (\grad^fV_1)^2\\
	&= - \frac32 \frac{ \psi }{ (1+\psi^2)^{\frac12} } \de_\tau (f^2) (\grad^fV_1)^2
		+ 3 \frac{ \psi^2 }{ (1+\psi^2)^{\frac12} } (\grad^fV_1)^2 .
	\end{align*}
	\begin{align*}
	\ref{C} &=
	- 2 \frac{ \psi }{ (1+\psi^2)^{\frac12} } 
		( 2 \psi \grad^fV_1 + f \Delta^fV_1 ) \de_\eta V_1 \\
	&= - 2 \frac{ \psi }{ (1+\psi^2)^{\frac12} } 
		( 2 \psi \grad^fV_1 + f \Delta^fV_1 ) (\grad^fV_1 - f\de_\tau V_1) \\
	&= - 4 \frac{ \psi^2 }{ (1+\psi^2)^{\frac12} } ( \grad^fV_1 )^2 
		+ 4 \frac{ \psi^2 }{ (1+\psi^2)^{\frac12} }f\grad^fV_1\de_\tau V_1 + \\
	&\qquad\qquad	- 2 \frac{ \psi }{ (1+\psi^2)^{\frac12} } f\Delta^fV_1 \grad^fV_1
		+ 2 \frac{ \psi }{ (1+\psi^2)^{\frac12} } f^2 \Delta^fV_1\de_\tau V_1 .
	\end{align*}
	We have two particular terms in this expression:
	\begin{align*}
	\textcircled{$\alpha$} &:=
	4 \frac{ \psi^2 }{ (1+\psi^2)^{\frac12} }f\grad^fV_1\de_\tau V_1 
	+ 2 \frac{ \psi }{ (1+\psi^2)^{\frac12} } f^2 \Delta^fV_1\de_\tau V_1 \\
	&= 2 \frac{ \psi }{ (1+\psi^2)^{\frac12} } \grad^f(f^2\grad^fV_1) \de_\tau  V_1  \\
	&= - 2 \frac{ \psi }{ (1+\psi^2)^{\frac12} } (f^2\grad^fV_1) \grad^f\de_\tau  V_1
		-  2 \de_\tau f \frac{ \psi }{ (1+\psi^2)^{\frac12} } (f^2\grad^fV_1) \de_\tau  V_1 \\
	&= - 2 \frac{ \psi }{ (1+\psi^2)^{\frac12} } f^2\grad^fV_1 
		(\grad^f\de_\tau V_1 + \de_\tau f\de_\tau V_1)
		 \\
	&= - 2 \frac{ \psi }{ (1+\psi^2)^{\frac12} } f^2\grad^fV_1 \de_\tau \grad^fV_1  \\
	&= -  \frac{ \psi }{ (1+\psi^2)^{\frac12} } f^2\de_\tau (\grad^fV_1)^2 \\
	&= \de_\tau \left(\frac{ \psi }{ (1+\psi^2)^{\frac12} }\right) f^2(\grad^fV_1)^2
		+ \frac{ \psi }{ (1+\psi^2)^{\frac12} } \de_\tau (f^2) (\grad^fV_1)^2
	\end{align*}
	and
	\begin{align*}
	\textcircled{$\beta$}&:=
	- 2 \frac{ \psi }{ (1+\psi^2)^{\frac12} } f\Delta^fV_1 \grad^fV_1 \\
	&= -  \frac{ \psi }{ (1+\psi^2)^{\frac12} } f \grad^f(\grad^fV_1 )^2 \\
	&= \frac{ \psi^2 }{ (1+\psi^2)^{\frac12} } (\grad^fV_1 )^2 
		+  \frac{ \psi }{ (1+\psi^2)^{\frac12} } f \de_\tau f (\grad^fV_1 )^2 \\
	&= \frac{ \psi^2 }{ (1+\psi^2)^{\frac12} } (\grad^fV_1 )^2 
		+  \frac{ \psi }{ (1+\psi^2)^{\frac12} }  \frac{\de_\tau (f^2)}{2} (\grad^fV_1 )^2 .
	\end{align*}
	Therefore:
	\begin{multline*}
	\ref{C} = 
	- 4 \frac{ \psi^2 }{ (1+\psi^2)^{\frac12} } ( \grad^fV_1 )^2 
		+ \textcircled{$\alpha$} + \textcircled{$\beta$} \\
	= - 3 \frac{ \psi^2 }{ (1+\psi^2)^{\frac12} } ( \grad^fV_1 )^2
		+ \de_\tau \left(\frac{ \psi }{ (1+\psi^2)^{\frac12} }\right) f^2(\grad^fV_1)^2 + \hfill \\ \hfill
		+ \frac32 \frac{ \psi }{ (1+\psi^2)^{\frac12} } \de_\tau (f^2) (\grad^fV_1)^2 .
	\end{multline*}
	Putting this together,
	\begin{align*}
	\ref{B}+\ref{C}
	&= \de_\tau \left(\frac{ \psi }{ (1+\psi^2)^{\frac12} }\right) f^2(\grad^fV_1)^2 \\
	&= \de_\tau \left(\frac{ \psi }{ (1+\psi^2)^{\frac12} }\right) (\grad^f(fV_1)-\psi V_1)^2 \\
	&= \de_\tau \left(\frac{ \psi }{ (1+\psi^2)^{\frac12} }\right) 
		( (\grad^f(fV_1))^2 + (\psi V_1)^2 - 2\grad^f(fV_1)\psi V_1 ) \\
	&= \de_\tau \left(\frac{ \psi }{ (1+\psi^2)^{\frac12} }\right) 
		( (\grad^f(fV_1))^2 - (\psi V_1)^2 - f\grad^f(\psi V_1^2) ) \\	
	&\overset{(*)}{=} \de_\tau \left(\frac{ \psi }{ (1+\psi^2)^{\frac12} }\right) 
		( (\grad^f(fV_1))^2 - (\psi V_1)^2 + \grad^ff \psi V_1^2 ) \\
	&= \de_\tau \left(\frac{ \psi }{ (1+\psi^2)^{\frac12} }\right) 
		(\grad^f(fV_1))^2 .
	\end{align*}
	In $(*)$ we used formula~\eqref{eq01121606}.
\end{proof}
\begin{lemma}\label{lem11171447}
	\begin{equation*}
	\ref{D} = 
		\de_\tau \left(\frac{ \psi }{ (1+\psi^2)^{\frac12} }\right) (\grad^fV_2)^2 .
	\end{equation*}
\end{lemma}
\begin{proof}[Proof of Lemma~\ref{lem11171447}]
	\begin{align*}
	\ref{D} 
	&=	2 \frac{ \psi }{ (1+\psi^2)^{\frac12} }
		\Delta^fV_2 \de_\tau V_2 \\
	&\overset{(*)}{=} -2 \frac{ \psi }{ (1+\psi^2)^{\frac12} }
		\grad^fV_2 \de_\tau \grad^fV_2 \\
 	&= - \frac{ \psi }{ (1+\psi^2)^{\frac12} }
		 \de_\tau (\grad^fV_2)^2 \\
	&= \de_\tau \left(\frac{ \psi }{ (1+\psi^2)^{\frac12} }\right)
		 (\grad^fV_2)^2.
	\end{align*}
	In $(*)$ we used formula~\eqref{eq01121606}.
\end{proof}

\begin{lemma}\label{lem11171446}
	\begin{equation*}
	\ref{E}+\ref{F}+\ref{G}
	 = -2 \de_\tau \left( \frac{ \psi }{ (1+\psi^2)^{\frac12} }\right) \grad^f(fV_1) \grad^fV_2 .
	\end{equation*}
\end{lemma}
\begin{proof}
	\begin{align*}
	\ref{G} 
	&= 2 (1+\psi^2)^{\frac12}   (\de_\eta V_1 \de_\tau V_2 - \de_\tau V_1 \de_\eta V_2) \\
	&= 2 (1+\psi^2)^{\frac12} ((\grad^fV_1-f\de_\tau V_1) \de_\tau V_2 - \de_\tau V_1 (\grad^fV_2-f\de_\tau V_2)) \\
	&= 2 (1+\psi^2)^{\frac12} (\grad^fV_1 \de_\tau V_2 - \de_\tau V_1 \grad^fV_2) .
	\end{align*}
	\begin{align*}
	\ref{F} &=
	2 \frac{ \psi }{ (1+\psi^2)^{\frac12} }
		\left(
		\Delta^fV_2 \de_\eta  V_1
		+ (-2 \psi \grad^fV_1 - f \Delta^fV_1) \de_\tau V_2
		\right) \\
	&= 2 \frac{ \psi }{ (1+\psi^2)^{\frac12} }
		\left(
		\Delta^fV_2 (\grad^fV_1-f\de_\tau V_1)
		- 2 \psi \grad^fV_1 \de_\tau V_2 
		- f \Delta^fV_1\de_\tau V_2 
		\right) \\
	&= 2 \frac{ \psi }{ (1+\psi^2)^{\frac12} }
		\left(
		- f\de_\tau V_1 \Delta^fV_2
		- f \de_\tau V_2 \Delta^fV_1
		+\Delta^fV_2 \grad^fV_1
		- 2 \psi \grad^fV_1 \de_\tau V_2 
		\right) \\
	&\overset{(*)}{=} 2 \frac{ \psi }{ (1+\psi^2)^{\frac12} }
		\left(
		\psi \de_\tau V_1 \grad^fV_2 + f \de_\tau (\grad^f V_1)  \grad^fV_2 
		+\psi \de_\tau V_2 \grad^fV_1 
			+\right.	 \\ 
	&\qquad\qquad\qquad\left.	+ f \de_\tau (\grad^f V_2)  \grad^fV_1
		+\Delta^fV_2 \grad^fV_1
		- 2 \psi \grad^fV_1 \de_\tau V_2 
		\right) \\
	&= 2 \frac{ \psi }{ (1+\psi^2)^{\frac12} }
		\left(
		\psi \de_\tau V_1 \grad^fV_2 + f \de_\tau (\grad^f V_1)  \grad^fV_2 
			+\right.  \\ 
	&\qquad\qquad\qquad\left.	+ f \de_\tau (\grad^f V_2)  \grad^fV_1
		+\Delta^fV_2 \grad^fV_1
		-  \psi \grad^fV_1 \de_\tau V_2 
		\right) .
	\end{align*}
	In $(*)$ we used formula~\eqref{eq01121606}.
	\begin{multline*}
	\ref{F}+\ref{G}
	= 2 \frac{(1+\psi^2)}{(1+\psi^2)^{\frac12}} (\grad^fV_1 \de_\tau V_2 - \de_\tau V_1 \grad^fV_2) + \hfill\\ \hfill
		+2 \frac{ \psi }{ (1+\psi^2)^{\frac12} }
		\left(
		\psi \de_\tau V_1 \grad^fV_2 + f \de_\tau (\grad^f V_1)  \grad^fV_2 
			+\right.\hfill \\ \hfill\left.
		+ f \de_\tau (\grad^f V_2)  \grad^fV_1
		+\Delta^fV_2 \grad^fV_1
		-  \psi \grad^fV_1 \de_\tau V_2 
		\right) \\
	= 2 \frac{1}{(1+\psi^2)^{\frac12}} (\grad^fV_1 \de_\tau V_2 - \de_\tau V_1 \grad^fV_2) + \hfill\\ \hfill
		+2 \frac{ \psi }{ (1+\psi^2)^{\frac12} }
		\left(
		 f \de_\tau (\grad^f V_1\grad^fV_2) 
		+\Delta^fV_2 \grad^fV_1 
		\right) .
	\end{multline*}
	\begin{multline*}
	\ref{E}+\ref{F}+\ref{G}
	= \frac{ \psi }{ (1+\psi^2)^{\frac12} }
			\left(
			- 4  \Delta^f V_2  \grad^f V_1  
			- 2 \grad^f V_2  \Delta^f V_1
			\right) 
			+\ref{F}+\ref{G} \\
	= 2 \frac{ \psi }{ (1+\psi^2)^{\frac12} }
			\left(
			- \Delta^f V_2  \grad^f V_1  
			- \grad^f V_2  \Delta^f V_1 
			+ f \de_\tau (\grad^f V_1  \grad^fV_2 )
			\right) + 	\hfill\\ \hfill
		+2 \frac{1}{(1+\psi^2)^{\frac12}} (\grad^fV_1 \de_\tau V_2 - \de_\tau V_1 \grad^fV_2) \\
	= 2 \frac{ \psi }{ (1+\psi^2)^{\frac12} }
			\left(
			- \grad^f( \grad^f V_2  \grad^f V_1 ) 
			+ f \de_\tau (\grad^f V_1\grad^fV_2)  
			\right) +\hfill\\ \hfill
		+2 \frac{1}{(1+\psi^2)^{\frac12}} (\grad^fV_1 \de_\tau V_2 - \de_\tau V_1 \grad^fV_2) \\
	= 2 \frac{ \psi }{ (1+\psi^2)^{\frac12} }  \de_\tau f \grad^f V_2  \grad^f V_1  
		- 2 \de_\tau \left( \frac{ \psi }{ (1+\psi^2)^{\frac12} }\right)f \grad^f V_1\grad^fV_2 + \hfill \\
		- 2 \frac{ \psi }{ (1+\psi^2)^{\frac12} } \de_\tau f \grad^f V_1\grad^fV_2 
		+2 \frac{1}{(1+\psi^2)^{\frac12}} (\grad^fV_1 \de_\tau V_2 - \de_\tau V_1 \grad^fV_2) \\
	= - 2 \de_\tau \left( \frac{ \psi }{ (1+\psi^2)^{\frac12} }\right)f \grad^f V_1\grad^fV_2 
		+2 \frac{1}{(1+\psi^2)^{\frac12}} (\grad^fV_1 \de_\tau V_2 - \de_\tau V_1 \grad^fV_2) . 
	\end{multline*}
	In particular, we have
	\begin{multline*}
	\frac{1}{(1+\psi^2)^{\frac12}} (\grad^fV_1 \de_\tau V_2 - \de_\tau V_1 \grad^fV_2)  
	= - \frac{1}{(1+\psi^2)^{\frac12}} \de_\tau\grad^f V_2  V_1
		+\hfill\\ \hfill
		+ \de_\tau \left(\frac{1}{(1+\psi^2)^{\frac12}}\right) \grad^fV_2 V_1
		+ \frac{1}{(1+\psi^2)^{\frac12}} \de_\tau \grad^fV_2 V_1 \\
	= \de_\tau \left(\frac{1}{(1+\psi^2)^{\frac12}}\right) \grad^fV_2 V_1 
	\end{multline*}
	and
	\[
	\de_\tau \left(\frac{1}{(1+\psi^2)^{\frac12}}\right)
	= - \frac{1}{(1+\psi^2)^{\frac32}} \psi \de_\tau \psi
	= - \psi \de_\tau \left(\frac{\psi}{(1+\psi^2)^{\frac12}}\right) .
	\]
	Therefore
	\begin{multline*}
	\ref{E}+\ref{F}+\ref{G} = \\
	= - 2 \de_\tau \left( \frac{ \psi }{ (1+\psi^2)^{\frac12} }\right)f \grad^f V_1\grad^fV_2 
		+2 \frac{1}{(1+\psi^2)^{\frac12}} (\grad^fV_1 \de_\tau V_2 - \de_\tau V_1 \grad^fV_2) \\
	= -2 \de_\tau \left( \frac{ \psi }{ (1+\psi^2)^{\frac12} }\right) f \grad^f V_1\grad^fV_2 
		-2 \psi \de_\tau \left(\frac{\psi}{(1+\psi^2)^{\frac12}}\right) \grad^fV_2 V_1 \\
	= -2 \de_\tau \left( \frac{ \psi }{ (1+\psi^2)^{\frac12} }\right) \grad^f(fV_1) \grad^fV_2 .
	\end{multline*}
\end{proof}

\appendix
\section{Useful formulas}\label{subs11241610}
In the case $f\in\Co^\infty(\R^2)$,
the adjoint operator of $\grad^f$ is
\[
(\grad^f)^* = -\grad^f-\de_\tau f ,
\]
i.e., if $A,B\in\Co^\infty(\R^2)$ and one of them has compact support, then
\[
\int_{\R^2} A\cdot\grad^f B \dd\eta\dd\tau
= - \int_{\R^2} \Big[ \grad^fA\cdot B + \de_\tau f\cdot A\cdot B \Big] \dd\eta\dd\tau .
\]
Notice that, if $f$ is smooth, the following holds:
\begin{align*}
	\de_\eta &= \grad^f - f\de_\tau , \\
	\de_\tau\grad^f  &= \grad^f\de_\tau + \de_\tau f\de_\tau  .
\end{align*}
If $A,B,C\in\Co^\infty(\R^2)$ and one of them has compact support, then
\begin{equation}\label{eq01121606}
\int_{\R^2}  A\cdot \de_\tau B \cdot\grad^fC \dd\eta\dd\tau
=
-\int_{\R^2}  \left(\grad^fA \cdot\de_\tau B\cdot C + A\cdot \de_\tau\grad^f B\cdot C \right)  \dd\eta\dd\tau .
\end{equation}

\printbibliography
\end{document}